\documentclass[11pt]{amsart}

\usepackage[a4paper,inner=3.6cm,outer=3.6cm,top=3.7cm,bottom=3.7cm]{geometry}
\usepackage{amsfo
	nts}
\usepackage{amssymb}
\usepackage{amsmath}
\usepackage[]{graphicx}
\usepackage[all,cmtip,2cell]{xy}
\usepackage{comment}
\usepackage{array}
\usepackage{enumitem}
\usepackage{mathtools}
\usepackage{multirow}
\usepackage{rotating,tabularx}

\setlist[itemize]{leftmargin=*}
\setlist[enumerate]{leftmargin=*}

\usepackage[usenames,dvipsnames]{xcolor}

\def\co{\colon\thinspace}
\newcommand{\C}{\mathbb{C}}
\newcommand{\Z}{\mathbb{Z}}
\newcommand{\N}{\mathbb{N}}
\newcommand{\rk}{\mathrm{rk}\,}

\newcommand{\R}{\mathbb{R}}
\newcommand{\bd}{\partial}

\newcommand{\M}{\mathcal{M}}
\newcommand{\cL}{\mathcal{L}}

\newcolumntype{C}[1]{>{\centering\arraybackslash$}p{#1}<{$}}

\newcommand{\CO}{\mathcal{CO}}

\newcommand{\Fuk}{\mathcal{F}uk}
\newcommand{\ev}{\mathrm{ev}}

\newcommand{\D}{\mathcal{D}}
\renewcommand{\O}{\mathcal{O}}

\newcommand{\cd}{\cdot}
\newcommand{\pt}{\mathrm{pt}}
\newcommand{\I}{{\sc i}}
\newcommand{\II}{{\sc ii}}
\newcommand{\III}{{\sc iii}}
\newcommand{\IV}{{\sc iv}}
\newcommand{\IandII}{\textsc{i,ii}}
\renewcommand{\S}{{\sc s}}
\newcommand{\BS}{\mathcal{BS}}

\renewcommand{\Re}{\mathrm{Re}}
\renewcommand{\Im}{\mathrm{Im}}
\newcommand{\LL}{\mathbf{L}}
\newcommand{\mm}{\mathbf{m}}
\renewcommand{\L}{\mathcal{L}}

\newtheorem{theorem}{Theorem}[section]
\newtheorem{proposition}[theorem]{Proposition}
\newtheorem{lemma}[theorem]{Lemma}
\newtheorem{corollary}[theorem]{Corollary}
\theoremstyle{definition}

\newtheorem{hypothesis}[theorem]{Hypothesis}
\theoremstyle{remark}
\newtheorem{remark}{Remark}[section]
\newtheorem{example}[remark]{Example}

\numberwithin{equation}{section}

\begin{document}

\title[From SH to Lagrangian enumerative geometry]{From symplectic cohomology\\ to Lagrangian enumerative geometry}

\author{Dmitry Tonkonog}
\thanks{This work was partially supported by the Simons
	Foundation (grant \#385573, Simons Collaboration on Homological Mirror Symmetry), and by the Knut and Alice Wallenberg Foundation (project grant Geometry and Physics).}
\address{UC Berkeley, Berkeley CA 94720, United States}
\address{
	Uppsala University, 751~06 Uppsala, Sweden}

\begin{abstract}
We build a bridge between  Floer theory on open symplectic manifolds  and the enumerative geometry of holomorphic disks inside their Fano  compactifications, by detecting elements in symplectic cohomology which are mirror to Landau-Ginzburg potentials. 
We also treat the higher Maslov index versions of the potentials. 

We discover a relation between higher disk potentials and  symplectic cohomology rings of smooth anticanonical divisor complements (themselves conjecturally related to closed-string Gromov-Witten invariants), and explore several other applications to the geometry of Liouville domains.
\end{abstract}
\maketitle

\section{Introduction}

\subsection{Overview}
A  recurring topic in symplectic geometry, closely related to mirror symmetry, is the interplay between the symplectic topology of open symplectic manifolds and that of their compactifications.  The idea is that Floer theory on open manifolds is usually easier to understand, and Floer theory on the compactification can sometimes be seen as a \emph{deformation} of the former one. 
This interplay  has been the driving force behind many important developments, notably the proofs of homological mirror symmetry for the genus 2~curve and the quartic surface by Seidel~\cite{Sei02,Sei11,Sei15} and for projective hypersurfaces by Sheridan \cite{She13,She15}; the works of Cieliebak and Latschev~\cite{CL09}, Seidel~\cite{Sei16b,Sei16}, Ganatra and Pomerleano \cite{GP16}; and the ongoing work of Borman and Sheridan~\cite{BS16}.

Our aim is to investigate this circle of ideas in a new setting that reveals a different aspect of mirror symmetry. Inside closed manifolds, we are interested in the enumerative geometry of holomorphic disks with boundary on a given (monotone) Lagrangian submanifold $L$; more precisely, in a specific and most basic invariant called the \emph {Landau-Ginzburg potential} and its important generalisation explained later.
Our main result links it to the wrapped symplectic geometry of open Liouville subdomains $M$ containing $L$, translating the Landau-Ginzburg potential into the language of \emph{symplectic cohomology} and \emph{closed-open string maps} on $M$. The result (Theorem~\ref{th:factor}) can be written as follows:
$$
W_L=\CO_L(\BS)
$$
where $W_L$ is the potential, $\BS\in SH^0(M)$ is a deformation element called the Borman-Sheridan class (explored by Seidel \cite{Sei16} in the different context of Calabi-Yau manifolds), and $\CO_L$ is the closed-open map. The precise statement appears later in the introduction. This theorem has a very clear mirror-symmetric interpretation which is the next thing we discuss, at a slightly informal level.

\subsection{Mirror symmetry context} Let $X$ be a (smooth compact) Fano variety of complex dimension $n$ equipped with a monotone K\"ahler symplectic form, and $\Sigma\subset X$ be a smooth anticanonical divisor. Roughly speaking, it is expected that the mirror $\check X$ of $X\setminus\Sigma$  is a variety which carries a proper map 
$$W\co \check X\to \C,$$ therefore its ring of regular functions is isomorphic to the polynomial ring in one variable $r$:
$$
\C[\check X]\cong \C[r],
$$
generated by the function $W$. The pair $(\check X,W)$ is called a \emph{Landau-Ginzburg model} and is mirror to the compact Fano variety $X$.

Suppose that $X\setminus \Sigma$ carries an SYZ fibration which produces the above mirror.
Consider a monotone Lagrangian torus $L\subset X$ which is exact in $X\setminus\Sigma$, and assume for the purpose of this subsection that it is an SYZ fibre. The space of $\C^*$-local systems on $L$ gives rise to an associated chart in the mirror:
$$
(\C^*)^n_L\subset \check X.
$$
Our results, 
Theorem~\ref{th:factor} and Proposition~\ref{prop:bs_antican}, can be interpreted as the following identity:
$$
W|_{(\C^*)^n_L}=W_L.
$$
In other words, start with the canonical regular function $W$ on $\check X$ (the LG model) and restrict it to a $(\C^*)^n_L$-chart; the result can be written as a Laurent polynomial. The claim is that this Laurent polynomial is the LG potential of the torus $L$. It means that the LG potentials of monotone Lagrangian tori in $X$ (which are exact in $X\setminus\Sigma$) are in fact different `avatars' of the same function $W$ defined on the whole mirror.

Furthermore, suppose
$M\subset X\setminus\Sigma$ is a Liouville subdomain, and assume that it is the preimage of a subset of the base of the SYZ fibration. In this setting, one expects an inclusion of the mirrors:
$$
\xymatrix{
	M\ar@{}[d]^(.25){}="a"^(.75){}="b" \ar@{^{(}->} "a";"b" \ar@{<-->}[r]&
	\check M \ar@{}[d]^(.25){}="a"^(.75){}="b" \ar@{^{(}->} "a";"b" 
	\\
	X\setminus\Sigma
	\ar@{<-->}[r] 
	&
	\check X
}
$$
As a general prediction, rings of regular functions are mirror to degree zero symplectic cohomology. (This was proven by  Ganatra and Pomerleano \cite{GP17} for the complement to the anticanonical divisor itself, see Section~\ref{sec:higher_disk}.) So one expects:
$$
SH^0(M)\cong \C[\check M].
$$
These rings can be complicated (much bigger than $\C[r]$), but they carry a distinguished element, the restriction of $W$:
$$
W|_{\check M}\in \C[\check M].
$$
It is natural to ask what is the symplectic counterpart, or the mirror, of this element.
This is answered by Theorem~\ref{th:factor},
which can be summarised in the language of mirror symmetry as follows:
$$
\textit{The Borman-Sheridan class}\ \ \BS\in SH^0(M)\ \  \textit{is mirror to} \ \ W|_{\check M}.
$$
The actual scope of Theorem~\ref{th:factor} is  broader: $\Sigma$ can have any degree, $L$ is not required to be an SYZ fibre or even a torus, and $M$ can be an arbitrary subdomain.

\subsection{Overview, continued}

Following an idea that we learned from James~Pascaleff we  introduce, in a restricted setting, the \emph{higher disk potentials} of a monotone Lagrangian submanifold which is disjoint from a smooth anticanonical divisor, and establish a similar theorem about them, roughly reading:
$$
W_{L,k}=\CO_L(\D_k)
$$
where $W_{L,k}$  is the Maslov index~$2k$ potential and $\D_k\in SH^0(M)$ is a deformation class. The case $k=1$ corresponds to the previous theorem.

Our results provide a convenient tool in the study of symplectic cohomology, holomorphic disk counts, and the topology of Liouville domains. To demonstrate this, we explore several applications:
\begin{itemize}
	\item A theorem relating higher disk potentials to the product structure on the symplectic cohomology ring of the complement $X\setminus \Sigma$ of a smooth anticanonical divisor $\Sigma$.
	
	This is of special interest in view of a conjecture of Gross and Siebert, cf.~\cite{GS16}. According to it, the symplectic cohomology product  can also be expressed in terms of closed-string log Gromov-Witten invariants of the pair $(X,\Sigma)$. Combined with this conjecture, our result would provide interesting identities between open- and closed-string GW theories of $X$ in line with the general intuition that specific combinations of holomorphic disks can be glued to holomorphic spheres.	
		
\item An alternative proof of a general form of the wall-crossing formula for Lagrangian mutations due to Pascaleff and the author \cite{PT17}.

\item A connection between seemingly distant properties of a Liouville domain $M$, under certain additional assumptions: the existence of a Fano compactification, the existence of an exact Lagrangian torus inside, the finite-dimensionality of $SH^0(M)$, and split-generation of $\Fuk(M)$ by simply-connected Lagrangians. Two sample outputs appear below.

\item We prove that Vianna's exact tori in del Pezzo surfaces minus an anticanonical divisor are not split-generated by spheres, which extends a result of  Keating \cite{Ke15}.
\end{itemize}

The proof strategies developed here are re-usable in different settings. Roughly speaking, they recast a stretching procedure for holomorphic curves, which would  typically be performed in the framework of Symplectic Field Theory \cite{CompSFT03}, within the world of symplectic cohomology. A major benefit is that it makes the moduli spaces unproblematically regular; and an equally important feature is that it gives  access to the algebraic structures like the closed-open maps. While we employ the Hamiltonian stretching procedure which by itself is entirely standard (it is used in classical Floer theory), the main content of the proofs lies in analysing the broken curves to show that, in some sense, they behave analogously to what one would expect under SFT stretching. This idea is implemented in the specific setup of Landau-Ginzburg potentials, but has a wider outlook.

\subsection{Main results}
Suppose $X$ is a monotone symplectic manifold, and $L\subset X$ is a monotone Lagrangian submanifold. Unless otherwise stated, we shall assume that $X$ is closed; and all Lagrangian submanifolds are closed and carry a fixed orientation and spin structure. We begin with a brief reminder of the Landau-Ginzburg potential of $L$ and its relation to local systems.

The simplest enumerative geometry problem relative to $L$ is to count holomorphic disks with boundary on $L$ which are of Maslov index~2, and whose boundary passes through a specified point on $L$.  The answer to this problem can be packaged into a generating function called the {\it Landau-Ginzburg superpotential}, or simply the {\it potential}. It is a Laurent polynomial 
$$W_L\in\C[H_1(L;\Z)]\cong \C[x_1^{\pm 1},\ldots, x_m^{\pm 1}]$$ 
where $m=\dim H_1(L;\C)$. When $L$ is monotone, Maslov index~2 disks have minimal positive symplectic area, so the potential is invariant under the choice of a tame almost complex structure and of Hamiltonian isotopies of $L$. See Section~\ref{sec:prelim} for a more extended reminder.

Let $\rho$ be a (rank one) $\C^*$-local system on $L$, and denote the pair $\LL=(L,\rho)$.
One can view the superpotential as a function on the space of $\C^*$-local systems, meaning that one can evaluate $W_L(\rho)\in \C$ to a complex number: it counts the same holomophic disks as described above, but their count is  weighted using the monodromies of $\rho$ along the boundaries of the disks, and the result is  a number.
We denote $$
\mm^0_\LL=W_L(\rho)\cdot 1_L\in HF^0(\LL,\LL)
$$
where $1_L$ is the Floer cohomology unit. (This is the \emph{curvature} of $\LL$ in the monotone Fukaya category of $X$.)

We are ready to state our main theorem; we shall use the notion of a Donaldson hypersurface which is reminded in Section~\ref{sec:prelim}, and a technical notion of a grading-compatible embedding which is defined after the statement. The theorem roughly asserts that the symplectic cohomology of any (nice) Liouville domain $M\subset X$ lying away from a Donaldson hypersurface $\Sigma\subset X$ has a canonical deformation class, the \emph{Borman-Sheridan class}, with the following property. For any 
monotone Lagrangian submanifold $L\subset X$ which happens to be contained in $M$ and becomes exact therein, the potential of $L$ can be computed by applying the closed-open map to this deformation class.

\begin{theorem}
	\label{th:factor}
	Let $X$ be a monotone symplectic manifold, and $\Sigma\subset X$ be a Donaldson hypersurface dual to $dc_1(X)$ for some $d\in\N$.  
	Fix a Liouville subdomain $M\subset X\setminus\Sigma$ such that $c_1(M)=0\in H^2(M;\Z)$ and the embedding is grading-compatible. There exists an element $\BS\in SH^0(M)$ which is called the Borman-Sheridan class, with the following property.
	
	Consider any monotone Lagrangian submanifold $L\subset X$ such that $L$ is contained in $M$ and is exact in $X\setminus \Sigma$ (automatically, $L$ is also exact in $M$). Let $W_L$ be the superpotential of $L$ computed inside $X$. Take a local system $\rho$ on $L$ and denote $\LL=(L,\rho)$. Then the image of $\BS$ under the closed-open map to $\LL$ equals $d\cdot  \mm^0_\LL=d\cdot  W_L(\rho)\cdot 1_L$:
	\begin{equation}
	\label{eq:diagram_M_and_CO}
	\xymatrix@R=12pt{
		\BS
		\ar@{|->}^-{}[r]
		\ar@{}|{\rotatebox[origin=c]{-90}{$\in$}}[d]
		&
		d\cdot \mm^0_\LL
		\ar@{}|{\rotatebox[origin=c]{-90}{$\in$}}[d]
		\\
		SH^0(M)\ar^-{\CO}[r]
		&
		HF^0_M(\LL,\LL)\\
	}
	\end{equation}
	We point out that the Floer cohomology $HF^0_M(\LL,\LL)\cong H^0(L)=\C\cdot 1_{L}$ is computed inside $M$.
	The Borman-Sheridan class depends on $M$ and its Liouville embedding into $X\setminus\Sigma$, but not on $L\subset M$.
\end{theorem}

The name for the Borman-Sheridan originates from the ongoing work \cite{BS16}; it has also  recently appeared in the work of Seidel \cite{Sei16} in the Calabi-Yau setup. We  present our version of the definition during the proof: specifically, in Proposition~\ref{prop:bs_def}.
In Section~\ref{sec:liouv} we collect several   applications of the result to the symplectic topology of Liouville domains; they were mentioned above.

\begin{remark}
	\label{rmk:CO_BS}
The theorem should be compared with the lemma of Auroux, Kontsevich and Seidel \cite[Proposition~6.8]{Au07}. If $\CO_X\co QH^*(X)\to HF^*(\LL,\LL)$ is the closed-open map computed in $X$, the lemma says that $\CO_X([\Sigma])=d\cdot \mm^0_\LL$. Theorem~\ref{th:factor} factors this relation through symplectic cohomology.
\end{remark}

\begin{remark}
	\label{rem:donaldson}
	If $\Sigma$
 is anticanonical ($d=1$),
 it is allowable to take $M=X\setminus\Sigma$ in Theorem~\ref{th:factor}. In this important case, one can compute the Borman-Sheridan class explicitly: see Proposition~\ref{prop:bs_antican} below.  Theorem~\ref{th:factor} need not apply to $M=X\setminus \Sigma$ 
	if $d>1$, because $c_1(X\setminus\Sigma;\Z)$ is torsion but may be non-zero in this case. 
\end{remark}

\begin{remark}
	\label{rmk:don_div}
By a version of the Donaldson or the Auroux-Gayet-Mohsen theorem \cite{AGM01}, \cite[Theorem~3.6]{CW17}, given a monotone Lagrangian submanifold $L\subset X$, one can find a Donaldson hypersurface $\Sigma$ away from $L$ and such that $L\subset X\setminus \Sigma$ is exact,  bringing us closer to the setup of Theorem~\ref{th:factor}. For example, in the K\"ahler case when $\Sigma$ is a complex divisor, one takes the K\" ahler form on $X\setminus \Sigma$ to be $-dd^c\log\|s\|$, for a section $s$ of $\O(\Sigma)$ such that $\Sigma=s^{-1}(0)$, and we take $-d^c\log\|s\|$ to be the natural primitive 1-form on $X\setminus \Sigma$, with respect to which $L$ has to be exact. In the general symplectic case, one uses the same form, but $s$ is now almost-holomorphic. It is explained in \cite{CW17} how to choose $\Sigma$ so that $L$ becomes exact in the complement, in the general symplectic case. However, the degree of this $\Sigma$ may in  general be large; see  also the discussion in~\cite{PT17}.
\end{remark}

\begin{remark}
Suppose
$$M_1\subset M_2\subset X\setminus\Sigma$$ are nested embeddings of Liouville domains, and consider the Viterbo map
$$
\mathrm{Vit}\co SH^*(M_2)\to SH^*(M_1).
$$
Using the tools developed in the proof of Theorem~\ref{th:factor}, it easy to argue that the Viterbo map respects the Borman-Sheridan classes. In particular, if $\Sigma$ is anticanonical, the Borman-Sheridan class of a Liouville subdomain $M\subset X\setminus \Sigma$ is the Viterbo map image of the Borman-Sheridan class of $X\setminus\Sigma$ itself, which is determined in Proposition~\ref{prop:bs_antican}.

For example, suppose that $M_1\cong T^*L$ is a Weinstein neighbourhood of a Lagrangian torus satisfying the conditions of Theorem~\ref{th:factor}, and $M_2=M$ is chosen as in Theorem~\ref{th:factor}. Then $SH^*(M_1)\cong \C[H_1(L;\Z)]\cong \C[x_1^{\pm 1},\ldots,x_m^{\pm 1}]$, and applying the composition
$$
SH^0(M)\xrightarrow{\mathrm{Vit}} SH^0(M_1)\cong \C[x_1^{\pm 1},\ldots,x_m^{\pm 1}]
$$
to $\BS$ computes $d\cdot W_L$ as a Laurent polynomial. An analogous statement holds for a Lagrangian $L$ of general topology, with an extra detail which we will encounter in the beginning of Section~\ref{sec:appl}. This way, Theorem~\ref{th:factor} admits an equivalent reformulation using Viterbo maps instead of the closed-open maps.
\end{remark}

Let us explain the grading conventions that are being used. The Floer cohomology $HF^*(\LL,\LL)$ is graded in the way singular cohomology is, with the unit in degree 0. Our definition of symplectic cohomology and its grading follows e.g.~the conventions of Ritter~\cite{Ri13}: the unit has degree 0, the closed-open maps have degree 0, and the Viterbo isomorphism \cite{Vi96,ASc06,ASc10,Ab15} reads
\begin{equation}
\label{eq:viterbo}
SH^*(T^*L)\cong H_{n-*}(\L L)
\end{equation}
where $\L L$ is the free loop space of a spin manifold $L$. Since Theorem~\ref{th:factor} assumes that $c_1(M)=0$, $SH^*(M)$ is $\Z$-graded. 

We come to the definition of what it means for $M\subset X\setminus\Sigma$ to be \emph{grading-compatible}. Let $d$ be the degree of $\Sigma$ so that $[\Sigma]$ is Poincar\'e dual to $dc_1(X)$. Then the $d$th power of the canonical bundle $(K_{X\setminus\Sigma})^d$ has a natural trivialisation $\eta$. If $\Sigma$ is given as the zero-set of a section $s$ of $K_X^{-d}$, then $\eta$ is the restriction of $s^{-1}$. Following \cite{PT17}, one says that $M\subset X\setminus\Sigma$ is grading-compatible if $\eta$ admits a $d$th root over $M$; that root provides a trivialisation of $K_M$ which is used to grade $SH^*(M)$. Grading-compatibility is a mild topological condition which is equivalent to the fact that $[\Sigma]$ is divisible by $d$ in $H_{2n-2}(X\setminus M;\Z)$. It is satisfied in all interesting examples, in particular when $\Sigma$ is anticanonical; see \cite[Section~2.2, Proposition~2.5]{PT17} for further discussion.

In Section~\ref{sec:antican}, we provide a version of Theorem~\ref{th:factor} for partial compactifications, or equivalently for normal crossings divisors instead of smooth ones. A rough sketch appears below.

\begin{theorem}[=Theorem~\ref{th:factor_red}]
When $L$ is disjoint from a normal crossings anticanonical divisor $\Sigma=\cup_{i\in I} \Sigma_i$ in a compact Fano variety $Y$, 
there holds a version of Theorem~\ref{th:factor} concerning the potential of $L$ in the non-compact manifold $X=Y\setminus\cup_{i\in J}\Sigma_i$, for any subset $J\subset I$.
\end{theorem}

\begin{remark}
The Borman-Sheridan classes in Theorem~\ref{th:factor_red} are essentially the first-order deformation classes from the work of Sheridan \cite{She17}.
\end{remark}

In Section~\ref{sec:higher_disk}, we introduce  \emph{higher disk potentials} 
$$W_{L,k}\in\C[x_1^{\pm 1},\ldots, x_m^{\pm 1}]$$
which, roughly speaking, count Maslov index $2k$ disks with boundary on $L$, passing through a specified point of $L$,  intersecting the given anticanonical divisor $\Sigma$ at a single point, and this intersection happens with order of tangency $k$, i.e.~intersection multiplicity $k$. 
The higher disk potentials depend on the choice of $\Sigma$ (at least, we do not show that they do not).

We show in Section~\ref{sec:higher_disk} that higher disk potentials can be used to compute the structure constants of the symplectic cohomology ring $SH^0(X\setminus \Sigma)$ with respect to its \emph{canonical basis} formed by (perturbations of) the periodic orbits of the standard $S^1$-periodic Reeb flow around $\Sigma$. We postpone a detailed discussion to Section~\ref{sec:higher_disk}, and now mention a sketch of the main statement:

\begin{theorem}[$=$Corollary~\ref{cor:struc_coef}]
It holds that	
$$
(W_L)^k=W_{L,k}+\sum_{0\le i\le k-1}c_{i,k}W_{L,i},
$$
where $c_{i,k}$ are the structure coefficients computing the product on $SH^0(X\setminus\Sigma)$ with respect to the canonical basis.
\end{theorem}

According to the mentioned  conjecture of Gross and Siebert, the same structure constants can be expressed in terms of certain closed-string log Gromov-Witten invariants of $(X,\Sigma)$, cf.~\cite{GS16,ArThesis}.

\subsection{Proof idea}

\begin{figure}[h]
	\includegraphics{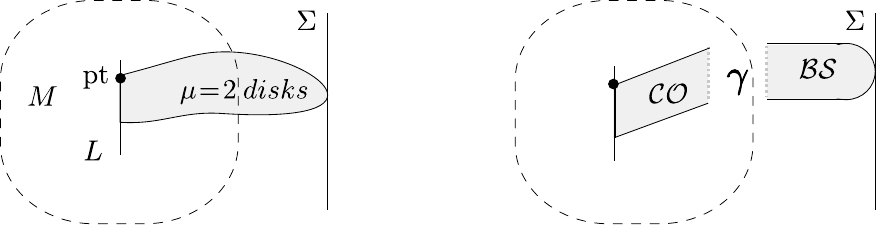}
	\caption{By ``domain-stretching'' the perturbation of the holomorphic equation, one makes holomorphic disks break this way in order to prove Theorem~\ref{th:factor}.}
	\label{fig:two_lags}
\end{figure}

Let us summarize the main aspects of the proof of Theorem~\ref{th:factor}. Assume that $L\subset M\subset X$ is a Lagrangian satisfying the conditions of Theorem~\ref{th:factor}.
Consider holomorphic Maslov index~2 disks in $X$ with boundary on $L$, satisfying a boundary point constraint. See Figure~\ref{fig:two_lags}, where we have depicted the Donaldson divisor $\Sigma$ which the disks will have to intersect.

In the proof of
Theorem~\ref{th:factor}, we perform a \emph{Hamiltonian domain-stretching} procedure upon the disks which makes them break into curves  shown in Figure~\ref{fig:two_lags}. The curves are broken along a periodic orbit of an \emph{\S-shaped} Hamiltonian which we explicitly choose. The count of the left parts of the broken curves in Figure~\ref{fig:two_lags} is precisely the closed-open map onto $L$.
It is crucial that the count of the right parts of the broken curves are independent on the Lagrangian $L$ we started with. Their count is what is taken to be the definition of the Borman-Sheridan class $\BS$.

\subsection*{Structure of the paper} 
In Section~\ref{sec:higher_disk} we discuss the higher disk potentials, a version of Theorem~\ref{th:factor} for them, and show that higher disk potentials can be used to compute structure constants of symplectic cohomology rings. 
In Section~\ref{sec:liouv}, we re-prove the wall-crossing formula for Lagrangian mutations due to Pascaleff and the author \cite{PT17}, and explore several applications to the existence of exact Lagrangian tori in Liouville domains, and split-generation by simply-connected Lagrangians.
In Section~\ref{sec:prelim}, we set up the preliminary material for the proof of Theorem~\ref{th:factor}.
In Section~\ref{sec:proofs} we prove Theorems~\ref{th:factor} and~\ref{th:factor_higher}. In Section~\ref{sec:antican} we compute the Borman-Sheridan class in the case when $M=X\setminus\Sigma$ and discuss what happens when $\Sigma$ is a normal crossings divisor, rather than a smooth one.

\subsection*{Related work} 
The recent works of Ganatra and Pomerleano \cite{GP17,GP18}, while they do not discuss Lagrangian enumerative geometry, use related methods and ideas. In order to get control on the broken curves arising in the proof of Theorem~\ref{th:factor}, we use  a collection of existing tools surrounding the theory of symplectic cohomology. Apart from the foundational material, these tools include confinement lemmas of Cieliebak and Oancea~\cite{CO18}, and a lemma about intersection multiplicity with a divisor at constant a periodic orbit by Seidel~\cite{Sei16}.
 Certain ingredients of our proof are also found in other recent papers on related subjects, for example: Diogo~\cite{DTh},  Gutt~\cite{Gu17}, Lazarev~\cite{La16}, Sylvan~\cite{Sy16}, and the work in preparation of Borman and Sheridan \cite{BS16}. 

\subsection*{Acknowledgements}This paper owes a lot to James Pascaleff; it uses several of his ideas and insights generously shared. I am thankful to Mohammed~Abouzaid, Denis~Auroux, Georgios Dimitroglou Rizell, Tobias~Ekholm, Sheel~Ganatra, Mark Gross, Yank\i\ Lekili, Paul~Seidel, Bernd Siebert, Ivan~Smith, Jack~Smith, Renato Vianna and the referee for useful  suggestions and interest. 

I am grateful to the following institutions for support  while this work was being developed:  Department of Pure Mathematics and Mathematical Statistics, University of Cambridge, and King's College, Cambridge (during my final PhD year); Mittag-Leffler Institute (during the Symplectic Geometry and Topology program, Fall 2015); Knut and Alice Wallenberg Foundation via the Geometry and Physics project grant; and the Simons Foundation via grant \#385573, Simons Collaboration on Homological Mirror Symmetry.

\section{Higher disk potentials and anticanonical divisor complements}
 \label{sec:higher_disk}
 \subsection{Higher disk potentials}
 A quick recollection of the usual Landau-Ginzburg potential is found in Section~\ref{sec:prelim}. 
 Assuming familiarity with it, we will now introduce higher disk potentials.
 
 Let $X$ be a closed monotone symplectic manifold, $L\subset X$ a monotone Lagrangian submanifold, $\Sigma\subset X$ a \emph{smooth} anticanonical divisor disjoint from $L$ and such that $L\subset \Sigma\setminus X$ is exact. Fix a point $p\in L$.
 Fix a tame almost complex structure $J$ preserving $\Sigma$, that is, $J(T_x\Sigma)=T_x\Sigma$ for each $x\in\Sigma$. 
 Consider a class $A\in H_2(X,L;\Z)$, a positive integer $k$ and define the moduli space
 $$
 \M(k,A)=\left\{
 \begin{array}{l}
 u\co (D^2,\bd D^2)\to (X,L),\ \bar \bd u=0, \ [u]=A,\\
 u(1)=p,\ u(0)\in \Sigma,\\
 u \text{ has a }k\text{-fold tangency to }\Sigma\text{ at }u(0).
 \end{array}
 \right\}
 $$
 The latter condition means that the local intersection number $u\cdot \Sigma$ at the point $u(0)$ equals $k$. For example, $k=1$ means transverse intersection. 

 To exclude multiply-covered disks, we assume that $J$ is domain-dependent, and the domain-dependence is restricted for convenience to both of: a neighbourhood of $1$ in $D$ and a neighbourhood of $p$ in $X$.
 
 \begin{remark}
 The homological intersection number of a Maslov index~$2k$ disk with $\Sigma$ equals $k$ by (\ref{eq:mu_and_intersec}) below. Due to positivity of intersections, Maslov index~$2k$ holomorphic disks having an order $k$ tangency to $\Sigma$ are precisely the ones that have a unique geometric intersection point with $\Sigma$.
 \end{remark}
 
  For $J$ generic outside a neighbourhood of $\Sigma$, $\M(k,A)$ is a manifold of dimension
 $$
 \dim \M(k,A)=\mu(A)-2k,
 $$
 see \cite[Lemma~6.7]{CM07}, \cite[Section~4.3]{GP16}.
 For a $\C^*$-local system $\rho$ on $L$, one would like to define the higher disk potential $W_{L,k}$ by
 $$
 W_{L,k}(\rho)=\sum_{A\, :\, \mu(A)=2k}\rho([\bd A])\cdot \#\M(k,A)\ \in\ \C,
 $$
 where $[\bd A]\in H_1(L;\Z)$ is the boundary of $A$,  $\rho([\bd A])\in\C^*$ is the value of the monodromy of $\rho$, and $\#\M(k,A)$ is the signed count of the points in the (oriented zero-dimensional) moduli space. Alternatively, fixing a basis for $H_1(L;\Z)/\mathrm{Torsion}$, one can package the same information into a Laurent polynomial
 $$
 W_{L,k}\in\C[x_1^{\pm 1},\ldots x_m^{\pm 1}]
 $$
 given by
 \begin{equation}
 \label{eq:higher_potential_def}
 W_{L,k}({\mathbf x})=\sum_{A\, :\, \mu(A)=2k} \mathbf{x}^{[\bd A]}\cdot \#\M(k,A),
 \end{equation}
 where  $\mathbf{x}^l=x_1^{l_1}\ldots x_m^{l_m}$ and $[\bd A]\in \Z^m=H_1(L;\Z)/\mathrm{Torsion}$.
 We have the following invariance lemma.

 \begin{lemma}
 	\label{lem:higher_bubble_sigma}
 The higher disk potential $W_{L,k}$ does not depend on the choice of $J$ preserving $\Sigma$, on $p\in L$, and is invariant under Hamiltonian isotopies of $L$ preserving $\Sigma$.
 \end{lemma}

\begin{proof}
The space of tame almost complex structures preserving $\Sigma$ is clearly connected. Fix a generic 1-parametric family $J_t$ of such, interpolating between the two given endpoints.
 To show the independence on $J$, one has to prove that the disks in $\M(k,A)$ do not bubble within the family $J_t$.

 Disk bubbles cannot occur because such bubble configurations would have intersection number with $\Sigma$ greater than $k$. Indeed, one gets the contribution of at least $k$ to this intersection number from the main disk inheriting the tangency, and all other disk bubbles have strictly positive intersection number with $\Sigma$.
  Next, standard tools using monotonicity imply that any bubbling which does not involve sphere components in $\Sigma$ is a condimension~2 phenomemon, so does not occur in a generic family $J_t$. 
  
  Now suppose that a stable bubble involves sphere components in $\Sigma$. Let $C\in H_2(\Sigma;\Z)$ be the class of the union of those components, and $B=A-C\in H_2(X,L;\Z)$.
Because $\Sigma$ is anticanonical, $c_1(X)\cdot C=[\Sigma]\cdot C=d$ for some $d\in \Z$. Then $\mu(B)=2k-2d$ and $[\Sigma]\cdot B=k-d$. It follows by positivity of intersections that, if we remove the sphere components from the bubble, the remaining part of the stable curve (in class $B$) has a unique geometric intersection with $\Sigma$, and the curve component containing that intersection is a disk  in the moduli space $\M(k-d,B)$, which has non-positive virtual dimension for every fixed $t$.
 Evaluating at the intersection point for all $t$ produces a chain of dimension at most 1 in $\Sigma$.

Recall that $\Sigma$ is anticanonical, so $c_1(\Sigma)=0$.
In particular, $c_1(\Sigma)\cdot C=0$, so the sphere components in $\Sigma$ sweep a subset of codimension 4 in $\Sigma$ for every $t$, by the index formula and the regularity of the simple curves. Hence they do not intersect the 1-chain appearing above, for all $t$ (if the $J_t$ are chosen generically), ruling out the bubbling.

Finally, it is standard that the invariance under Hamiltonian isotopies reduces to the invariance under the choice of $J$.
\end{proof}

 \begin{theorem}
 	\label{th:factor_higher}
 	Let $X$ be a closed monotone symplectic manifold, $\Sigma\subset X$ a smooth anticanonical divisor, and
 	let $M\subset X\setminus\Sigma$ any Liouville subdomain. There exists a class $\D_k\in SH^0(M)$ such that the following holds. 
 	
 	Take any monotone Lagrangian submanifold $L\subset X$ such that $L$ is contained in $M$ and is exact in $X\setminus \Sigma$. Let $W_{L,k}$ be the higher disk potential of $L$ computed inside $X$. Take a local system $\rho$ on $L$ and denote $\LL=(L,\rho)$. 
 	Then the image of $\D_k$ under the closed-open map to $\LL$ equals $W_{L,k}(\rho)\cdot 1_L$:
 	\begin{equation}
 	\label{eq:diagram_M_and_CO_higher}
 	\xymatrix@R=12pt{
 		\D_k
 		\ar@{|->}^-{}[r]
 		\ar@{}|{\rotatebox[origin=c]{-90}{$\in$}}[d]
 		&
 	 W_{L,k}(\rho)\cdot 1_L
 		\ar@{}|{\rotatebox[origin=c]{-90}{$\in$}}[d]
 		\\
 		SH^0(M)\ar^-{\CO}[r]
 		&
 		HF^0_M(\LL,\LL)\\
 	}
 	\end{equation}
 \end{theorem} 
 
 Clearly, $W_{L,1}=W_L$ is the usual Landau-Ginzburg potential, and
one easily sees from the proof that $\D_1=\BS$. 
 
 \subsection{Symplectic cohomology of anticanonical divisor complements}
Let $X$ be a monotone symplectic manifold and $\Sigma\subset X$ a smooth Donaldson divisor of degree $d$. We remind that this means $[\Sigma]=dc_1(X)$ for $d\in\Z_{>0}$. The Reeb flow on the boundary-at-infinity of  $X\setminus \Sigma$ is $S^1$-periodic, and the space of its Reeb orbits is homeomorphic to $S\Sigma$, the unit normal bundle to $\Sigma$. One can choose a Hamiltonian perturbation in such a way that at chain level, the complex computing $SH^*(X\setminus\Sigma)$ has a fractional $\tfrac 1 d\Z$-grading and is isomorphic, as a graded vector space, to:
\begin{equation}
\label{eq:CH_chain_level}
CF^*(X\setminus\Sigma)\cong C^*(X\setminus\Sigma)\oplus (C^*(S\Sigma)\otimes r\C[r]),\quad |r|=2-2/d.
\end{equation}

See e.g.~\cite[(3.19)]{GP16}, cf.~\cite[Lemma~3.4]{DTh}, \cite{Sei08}.
Now suppose that $d=1$, that is, $\Sigma$ is anticanonical. Then at chain level,
\begin{equation}
\label{eq:CH_chain_level_0}
CF^0(X\setminus\Sigma)\cong\C[r],\quad CF^{i}(X\setminus\Sigma)=0\textit{ for }i<0.
\end{equation}

The theorem below is due to Ganatra and Pomerleano \cite[Theorem~5.31]{GP18}, compare \cite{DTh}.

\begin{theorem}
	\label{th:gp}
If $\Sigma$ is smooth anticanonical, then the symplectic cohomology differential on $CF^0(X\setminus\Sigma)$ from (\ref{eq:CH_chain_level}) vanishes. Consequently, as a vector space,
$$
SH^0(X\setminus\Sigma)=CF^0(X\setminus\Sigma)
$$
has the natural geometric basis $\{1,r,r^2,\ldots\}$
coming from (\ref{eq:CH_chain_level}), where we denote the element $1\otimes r^i$ appearing in the right summand of (\ref{eq:CH_chain_level}) simply by $r^i$.

Moreover, as an algebra, $SH^0(X\setminus\Sigma)$ is generated by $r$ and is abstractly isomorphic to the polynomial algebra in one variable.\qed
\end{theorem}

We wish to add several comments about the theorem.
First, if the Floer differential vanishes on
$
CF^0(X\setminus\Sigma)
$,
it follows by (\ref{eq:CH_chain_level_0}) that
$
SH^0(X\setminus\Sigma)=CF^0(X\setminus\Sigma).
$

\begin{remark}
	This has to be compared	to the case when $\Sigma$ is a smooth symplectic divisor such that $[\Sigma]=\tfrac 1 d c_1(X)$. An example of this situation is when $\Sigma$ is a hyperplane in $\C P^n$; then $[\Sigma]=\tfrac 1{n+1}c_1(\C P^n)$. The symplectic cohomology of $$\C P^n\setminus\Sigma\cong \R^{2n}$$  vanishes. In this case a version of (\ref{eq:CH_chain_level}) holds with $|r|=1-d$, and the complex $CF^*(X\setminus\Sigma)$ now has elements of negative degree. Therefore the Floer differential can hit the unit.
\end{remark}	

We now wish to explain the \emph{`moreover'} part of Theorem~\ref{th:gp}. By a suitable choice of a  Hamiltonian perturbation realising the complex (\ref{eq:CH_chain_level}), one can arrange the actions of the generators $r^i$ appearing in Theorem~\ref{th:gp} to be additive up an error:
$$
A(r^{i+j})=A(r^i)+A(r^j)+(\epsilon\text{-error}),
$$ 
where the error can be made arbitrarily small, in particular separating the orbit actions. The chain-level differential and product on symplectic cohomology are action-non-decreasing. It follows that on chain level, the symplectic cohomology product $r^i*r^j$
is a combination of the elements $r^k$ for $k\le i+j$. (The actions of the $r^i$ are negative and tend  to $-\infty$.) Moreover, the coefficient of $r^{i+j}$ is computed by low-energy curves, and  choosing a suitable Hamiltonian (a convenient choice is to make it autonomous and work in the $S^1$-Morse-Bott setup, see e.g.~\cite{Bo03,BO09,BEE}) one can show that there is a unique such low-energy curve.

It follows that the product structure on $SH^0(X\setminus\Sigma)$, written in the canonical basis $\{1,r,r^2,\ldots\}$, has the following form:
\begin{equation}
\label{eq:str_const}
r^{*k}=r^k+\sum_{0\le i\le k-1}c_{i,k}r^i.
\end{equation}
Here we denote by $r^{*k}=r*\ldots *r$ the $k$th power of $r$ with respect to the symplectic cohomology product.
The \emph{structure constants} $c_{i,k}$ need not be trivial, and are of big interest as they are expected to encode certain closed log  Gromov-Witten invariants of $X$, see e.g.~\cite{GS16}, \cite[Remark~5.39]{GP18}. 

\subsection{Structure constants and higher potentials}
We take a different viewpoint and reveal that the structure constants above can be computed from the higher disk potentials of a monotone Lagrangian submanifold.
The next proposition will be proved in Section~\ref{sec:proofs}; it also follows from a  recent and more general result of Ganatra and Pomerleano~\cite[Theorem 4.31]{GP18}.

\begin{proposition}
	\label{prop:bs_antican}
	Let $\Sigma$ be a smooth anticanonical divisor and take $M=X\setminus\Sigma$.
In the situation of Theorem~\ref{th:factor}, it holds that $\BS=r$. In the  situation of Theorem~\ref{th:factor_higher}, it holds that $\D_k=r^k$.
\end{proposition}

\begin{lemma}
\label{lem:W_powers}
In the setting of Theorem~\ref{th:factor_higher} and Proposition~\ref{prop:bs_antican}, let $\rho$ be any local system on $L$. Then
$$
\CO(r^{*k})=(W_L(\rho))^k\cdot 1_{\LL},
$$
where $W_L$ is the usual LG potential of $L$ computed in $X$.
\end{lemma}

\begin{proof}
	One has $\CO(r)=W_L(\rho)
	\cdot 1_{\LL}$ by Theorem~\ref{th:factor} and Proposition~\ref{prop:bs_antican}. The claim follows from the fact that $\CO$ is a ring map. 
\end{proof}	

\begin{corollary}
	\label{cor:struc_coef}
	Suppose $\Sigma\subset X$ is an anticanonical divisor, and $L\subset X\setminus\Sigma$  is a Lagrangian submanifold which is monotone in $X$ and exact in $X\setminus\Sigma$.	
Assume that we are in a setting when the higher disk potentials $W_{L,i}$ are defined for all $i\le k$. Writing the disk potentials as Laurent polynomials, the following identity between Laurent polynomials holds:
$$
(W_L)^k=W_{L,k}+\sum_{0\le i\le k-1}c_{i,k}W_{L,i},
$$
where $W_L=W_{L,1}$ is the usual Landau-Ginzburg potential and $c_{i,k}$ are the structure constants from (\ref{eq:str_const}).\qed 
\end{corollary}

\begin{example}
Take the monotone Clifford torus $L\subset \C P^2$.  There exists a smooth elliptic curve $\Sigma\subset \C P^2$ such that $L$ is exact in its complement. One has, in some basis for $H_1(L;\Z)$:
$$
W_L=x+y+1/xy.
$$
One computes:
$$
(W_L)^3=6+(x^3+y^3+1/x^3y^3+3x^2y+3xy^2+3x/y+3/xy^2+3y/x+3/x^2y)
$$
It is expected by several methods that 
$$c_{0,3}=6,$$ while $c_{1,3}$, $c_{2,3}$ vanish for minimal Chern number reasons. It follows that $W_{L,3}=(W_L)^3-6$; compare with \cite[Remark~5.39]{GP18}.

Conversely, we hope that it is possible to compute the higher disk potentials of toric fibres relative to a smooth anticanonical divisor by other means, and deduce the structure constants for symplectic cohomology using that. The main point is that it is not hard to determine the holomorphic disks of any Maslov index intersecting the singular \emph{toric boundary divisor} at a single geometric point. However, the enumerative geometry changes when we pass to the smooth divisor, and one needs to understand this change.
It is left as a subject for future research.
\end{example}

\section{Wall-crossing and symplectic topology of Liouville domains}
\label{sec:liouv}
\subsection{The wall-crossing formula}
Theorem~\ref{th:factor} provides an alternative proof of a general form of the wall-crossing formula due to Pascaleff and the author~\cite{PT17}. Let $L_0,L_1\subset M\subset X$ be two Lagrangian submanifolds satisfying the conditions of Theorem~\ref{th:factor}. Let $\rho_i$ be a $\C^*$-local system on $L_i$ and denote $\LL_i=(L_i,\rho_i)$.

\begin{theorem}[\cite{PT17}]
	\label{th:pt}
	Assume that the Floer cohomology of the pair below computed in $M$ does not vanish:
	$$
	HF^*_M(\LL_0,\LL_1)\neq 0.
	$$
	Then it holds that
	$$
	W_{L_0}(\rho_0)=W_{L_1}(\rho_1)\in \C,
	$$		
	where the potentials are computed inside $X$.	
\end{theorem}

We refer to \cite{PT17} for the context surrounding this theorem, and its applications. In particular, we remind that in any given geometric setting, one still has to compute the pairs $(\rho_1,\rho_2)$ for which the Floer cohomology appearing in the statement does not vanish; this computation determines the precise shape of the wall-crossing formula relating the potentials.

Let us prove Theorem~\ref{th:pt} using Theorem~\ref{th:factor}. Consider the following diagram:
\begin{equation}
	\label{eq:CO_pair}
	\begin{array}{ccccc}
		d\cdot W_{L_0}(\rho_0)\cd 1_{L_0}
		&\reflectbox{$\longmapsto$}
		&
		\BS
		&\longmapsto
		&d\cdot  W_{L_1}(\rho_1)\cd 1_{L_1}
		\\
		\rotatebox[origin=c]{-90}{$\in$}
		&&
		\rotatebox[origin=c]{-90}{$\in$}
		&&
		\rotatebox[origin=c]{-90}{$\in$}\\
		HF_M^0(\LL_0,\LL_0)
		&\xleftarrow{\CO_{L_0}} &
		SH^0(M)
		&\xrightarrow{\CO_{L_1}}&
		HF_M^0(\LL_1,\LL_1)
	\end{array}
\end{equation}
Above, the subscripts for the $\CO$-maps are simply used to specify the target.
Now pick any non-zero element
$$
0\neq x\in HF^*_M(\LL_0,\LL_1),
$$
assuming that this Floer cohomology is non-vanishing. The next equality follows from the general fact that the closed-open maps turn Floer cohomologies into $SH^*(M)$-modules:
$$
\mu^2(\CO_{L_0}(\BS),x)=\mu^2(x,\CO_{L_1}(\BS)).
$$
Here $\mu^2$ denotes the product in Floer cohomology, read left-to-right (this convention is opposite to \cite{SeiBook08}).	
By (\ref{eq:CO_pair}), the equality above rewrites as
$$
d\cdot W_{L_0}(\rho_0)\cd x=d\cdot  W_{L_1}(\rho_1)\cd x.
$$
Since $x\neq 0$, it follows that   $W_{L_0}(\rho_0)=W_{L_1}(\rho_1)$. 
While the original proof of Theorem~\ref{th:pt} given in \cite{PT17} using relative Floer theory is easier and less technical than that of Theorem~\ref{th:factor}, we hope that the present proof serves as a useful illustration of Theorem~\ref{th:factor}.

\label{sec:appl}
\subsection{Lagrangian embeddings with constant potential}
We will now discuss the following hypothesis about a Fano manifold $X$, which will be used later.

\begin{hypothesis}
	\label{hyp:const}
Let $X$ be a monotone symplectic $2n$-manifold. Every monotone Lagranian torus in $X$ has non-constant LG potential.  
\end{hypothesis}

We begin with the following.

\begin{lemma}
The potential of any monotone Lagrangian torus has vanishing constant term.
In particular, if the potential is constant, then it is zero.
\end{lemma}

\begin{proof}
	Let $M$ be a Weinstein neighbourhood of $L$, and $\Sigma$ an auxiliary Donaldson divisor provided by Remark~\ref{rmk:don_div}, which guarantees that Theorem~\ref{th:factor} applies to $M$.
	Let $SH^*_+(M)$ be the positive symplectic cohomology, whose chain complex is generated by non-constant periodic orbits.
We employ Lemma~\ref{lem:non_const} below, whose equivalent formulation says that $\BS$ lies in the image of the map $SH^0_+(M)\to SH^0(M)$.  Note that $SH^0_+(T^*T^n)$ is generated at chain level by non-contractible loops in $T^*T^n$. So constant term of the potential, which is computed by disks with contractible boundary, vanishes. 
\end{proof}

\begin{remark}
To compare, the potential of a Lagrangian sphere is always a constant, since the sphere is simply-connected. This constant need not be zero.
\end{remark}

\begin{lemma}
Suppose $X$ has a non-trivial one-pointed descendant Gromov-Witten invariant $\langle \psi_d\, \mathrm{pt}\rangle$ for some $d\ge 2$; here $\mathrm{pt}$ is a generator of $H_0(X)$. Equivalently, the quantum period \cite[Section~B]{CCGK16} of $X$  is a series which is not identically one. Then Hypothesis~\ref{hyp:const} holds for $X$.
\end{lemma}

\begin{proof}
This follows from \cite[Theorem~1.1]{To18}.	
\end{proof}

To our knowledge, the above lemma, and hence Hypothesis~\ref{hyp:const}, is expected to hold for all Fano varieties, but this does seem to be currently proven.

\begin{corollary}
	\label{cor:hyp_comp_int}
Let $X$ be a Fano complete intersection in a smooth toric Fano variety. Then Hypothesis~\ref{hyp:const} holds for $X$.
\end{corollary}

\begin{proof}
The quantum period of $X$  can be explicitly computed by iteratively applying the quantum Lefschetz formula of Coates and Givental \cite{CoGi07}. The answer is written down in \cite[Corollary~D.5]{CCGK16}, and this quantum period sequence is non-zero (i.e.~the quantum period series is not identically one). To see this, it may be helpful to take the cohomology classes of curves into account. If the quantum periods of $X$ all vanish, the quantum Lefschetz formula implies that the quantum periods of the ambient toric variety $Y$ can be non-zero only for classes $\beta\in H_2(Y)$ which have the property that $c_1(X)\cdot \beta=[Y]\cdot \beta+1$ (that is, they come from Chern number~1 classes in $Y$), compare with the formula for $C(\tau)$ in the proof of \cite[Corollary~D.5]{CCGK16}. This is a contradiction because for toric Fano varieties, all effective classes contribute non-trivially to quantum periods, see e.g.~\cite[Corollary~C.2]{CCGK16}.

Strictly speaking, \cite{CCGK16} compute algebro-geometric GW~invariants, which are conjecturally equivalent to the symplectic ones. However, the work in preparation \cite{DTVW17} re-proves a version of the quantum Lefschetz formula in the context of symplectic GW~invariants.
\end{proof}	

Independently of the above, we can also prove Hypothesis~\ref{hyp:const}
in the case when $QH^*(X;\C)$ is semisimple.

\begin{lemma}
	\label{lem:diff_ls}
	Let $L\subset X$ be a monotone Lagrangian submanifold and $\rho_1\neq \rho_2$ two different local systems on $L$. Let $\LL_i=(L,\rho_i)$. Then $ HF^*(\LL_1,\LL_2)=0$.
\end{lemma}

\begin{proof}
	Recall that Oh's spectral sequence  \cite{Oh96A} converging to $HF^*(\LL_1,\LL_2)$	begins with the first page $C^*(L)$  carrying the singular differential of degree 1 twisted by the local system $\rho_1\rho_2^{-1}$. So it is enough to check that the twisted singular cohomology of $L$ vanishes. Consider a minimal Morse function on $T^n$ with $2^n$ critical points, and represent the Morse complex, as a vector space, by the exterior algebra $\Lambda^*\langle p_1,\ldots, p_n\rangle$ on the degree~1 Morse generators $p_i$. In this Morse model, the twisted differential is given  by
	$$
	a\mapsto a\wedge \left(\sum_{i=1}^n (\lambda_i-1)p_i\right)
	$$
	where $\lambda_i\in \C^*$ is the monodromy of $\rho_1\rho_2^{-1}$ around the loop corresponding to the class of $p_i$. The resulting complex is acyclic unless $\rho_1=\rho_2$.
\end{proof}	

\begin{corollary}
	\label{cor:ve}
If $X$ is a Fano variety with semisimple $QH^*(X;\C)$, 
 Hypothesis~\ref{hyp:const} holds for it.
\end{corollary}

\begin{proof}
	If $QH^*(X)$ is semisimple and its even part has rank $k$, the monotone Fukaya category of $X$ cannot have more than $k$ non-trivial, mutually orthogonal objects (meaning: every object has non-trivial self-Flor homology, and every pair of distinct objects has vanishing Floer homology). See \cite[(1d)]{Sei14}, \cite[Theorem~1.25]{EnPo09}. If $W_L$ is constant, for any local system $\rho$ and $\LL=(L,\rho)$ it holds that $HF(\LL,\LL)\neq 0$. In view of Lemma~\ref{lem:diff_ls}, this means that the Fukaya category of $X$ has infinitely many non-trivial, mutually orthogonal objects, which is a contradiction.
\end{proof}	

\subsection{Exact tori in Liouville domains}

Below is the main result of this section, which we will use to draw several corollaries about the symplectic topology of Liouville domains.

\begin{proposition}
	\label{prop:no_exact}
Suppose $X$ is a  monotone symplectic manifold, and $\Sigma\subset X$ is a Donaldson hypersurface.
Suppose $M\subset X\setminus \Sigma$ is a grading-compatible Liouville subdomain with $c_1(M)=0$, $L\subset M$ is an exact torus with vanishing Maslov class, and:
\begin{itemize}
	\item either $SH^0(M)$ is finite-dimensional as a vector space,
	\item or there exist countably many compact exact Lagrangian submanifolds $K_i\subset M$ with $H_1(K_i;\C)=0$ satisfying the following property: for any $\C^*$-local system $\rho$ on $L$, the object $(L,\rho)$ is split-generated, in the compact exact Fukaya category of $M$, by the $K_i$.  
\end{itemize}
Then the potential of the monotone torus  $L\subset X$ is constant. In particular, both cases provide a contradiction if Hypothesis~\ref{hyp:const} holds for $X$.
\end{proposition}

\begin{example}
Suppose $X$ is Fano and $[\Sigma]=dc_1(X)$ where $d\in \mathbb{Q}_{>0}$ and $d\neq 1$. Then $SH^0(M)$ is finite-dimensional by (\ref{eq:CH_chain_level}).
\end{example}

\begin{proof}
	Suppose $L\subset M$ is such a torus. Then it becomes monotone under the inclusion $M\subset X$, compare \cite[Example~3.2, Lemma~3.4]{CW17}, \cite[Section~2]{PT17}. We claim that either of the two given conditions imply that $W_L$ is constant, where the potential $W_L$ is computed in $X$. 
	
	Under the first condition, assume $W_L$ is a non-constant Laurent polynomial. Then the powers $(W_L)^k$ are linearly independent for all $k$. It follows from Theorem~\ref{th:factor} that the powers $\BS^k\in SH^0(M)$ with respect to the symplectic cohomology product are linearly independent for all $k$, hence $SH^0(M)$ is infinite-dimensional.
		
	Under the second condition, for any local system $\rho$ there exists a Lagrangian $K_i$ such that \begin{equation}
	\label{eq:hf_k}
		HF^*_M((L,\rho),K_i)\neq 0.
	\end{equation}	
	For a fixed $K_i$, the set of all $\rho\in (\C^*)^n$ such that (\ref{eq:hf_k}) holds is Zariski closed in $(\C^*)^n$. It follows that there exists a $K_i$ such that (\ref{eq:hf_k}) holds for all local systems, because the union of countably many Zariski closed sets cannot cover $(\C^*)^n$ unless one of the sets coincides with it.
	Since $H_1(K_i;\C)=0$, the LG potential of $K\subset X$ is automatically constant: $W_K\equiv c$. By the wall-crossing formula (Theorem~\ref{th:pt}),
	$$W_L(\rho)=c$$
	for all $\rho$. It means that $W_L\equiv c$. 
\end{proof}

In the rest of the section we present several applications of Proposition~\ref{prop:no_exact} to the symplectic topology of Liouville domains, offering a fresh look on some of the known results in the field and providing new extensions thereof.

\subsection{Generation by simply-connected Lagrangians}

\begin{corollary}
	\label{cor:not_gen}
Suppose $X$ is a monotone symplectic manifold containing a monotone torus with non-constant potential, and $\Sigma\subset X$ is a Donaldson divisor such that the torus is exact in its complement. Then the compact exact Fukaya category of any grading-compatible Liouville subdomain of $X\setminus\Sigma$ with vanishing Chern class is not split-generated by simply-connected Lagrangians.
\end{corollary}

\begin{proof}
The given hypotheses imply that $L\subset X\setminus\Sigma$ has vanishing Maslov class. Hence Proposition~\ref{prop:no_exact} applies.
\end{proof}

\begin{example}
Keating~\cite{Ke15} showed that the 4-dimensional Milnor fibres of isolated complex singularities of modality one contain an exact Lagrangian torus $L$ not split-generated by the vanishing Lagrangian spheres. She explicitly showed that the locus in $(\C^*)^2$ consisting of local systems $\rho$ such that the Floer cohomology between $(L,\rho)$ and the vanishing Lagrangian spheres is  1-dimensional. On the other hand, split-generation would force the locus to be 2-dimensional.

The simplest modality one Milnor fibres are
$$
\mathcal{T}_{3,3,3}=dP_6\setminus\Sigma,\quad \mathcal{T}_{2,4,4}=dP_7\setminus\Sigma,\quad \mathcal{T}_{2,3,6}=dP_8\setminus\Sigma,
$$
where $dP_i$ is the del Pezzo surface which is the blowup of $\C P^2$ at $i$ points, and $\Sigma$ is a smooth anticanonical divisor in each of them. Vianna showed \cite{Vi14,Vi17} that each of these Milnor fibres contains infinitely many exact Lagrangian tori, whose potentials in the compactification $dP_i$ are non-constant. In view of Corollary~\ref{cor:not_gen}, we arrive at the following generalisation of \cite{Ke15}.
\end{example}

\begin{corollary}
Fix any finite collection of Lagrangian spheres in the Milnor fibre $\mathcal{T}_{3,3,3}$, $\mathcal{T}_{2,4,4}$ or $\mathcal{T}_{2,3,6}$.
For almost all $\C^*$-local systems (precisely, for a non-empty Zariski open set of them) on each exact Vianna torus in the Milnor fibre, that torus is not split-generated by the fixed collection of  spheres considered as separate objects of the Fukaya category.\qed
\end{corollary}	

\begin{example}
Let $M$ be the plumbing of two copies of $T^*S^2$ at two points with the same intersection sign. It is Liouville deformation equivalent to
$$
\{(x,y,z)\in\C^3:xy=z^2-1\}\setminus\{xy=-1\}.
$$
This can be seen using the Lefschetz fibration onto $\C\setminus\{0\}$ by projecting to the $z$-plane.
Therefore $M$ embeds into the quadric $\C P^1\times\C P^1$ away from an anticanonical divisor (which can be smoothed).  The domain $M$ contains an exact torus (the matching torus for a loop in the base of the Lefschetz fibration encircling both singular points, and enclosing the correct amount of area to make the resulting torus exact).  Corollary~\ref{cor:not_gen}
implies that this torus is not generated by any collection of Lagrangian spheres in $M$. This is totally expected because the torus is displaceable from both core spheres.
\end{example}

\begin{remark}
A subtlety must be pointed out here. Suppose $X$  is a Fano projective hypersurface and $\Sigma\subset X$ is an anticanonical divisor.  
Sheridan proved \cite{She11,She13} that the compact exact Fukaya category of  $X\setminus\Sigma$ is split-generated by the union of certain Lagrangian spheres considered as a \emph{single immersed Lagrangian submanifold}. This is different from being split-generated by the collection of the same Lagrangian spheres considered separately: already in the previous example,  $\Fuk(M)$ is generated by the single object $K=K_1\cup K_2$, the union of the two core spheres. In particular, the matching exact torus $L\subset M$ is Hamiltonian displaceable from $K_1$ and $K_2$, but not from the union $K_1\cup K_2$.
\end{remark}

\subsection{Non-existence of exact tori}

We move on to the discussion along a different line: if it is known that $\Fuk(M)$ is split-generated by spheres, Proposition~\ref{prop:no_exact} may be invoked to show the non-existence of exact Lagrangian tori in $M$.
Current knowledge includes the following:
\begin{itemize}
	\item by Ritter \cite{Ri10}, four-dimensional $A_k$-Milnor fibres do not contain exact Lagrangian tori;
	\item by Abouzaid and Smith \cite{AS12}, $A_2$-Milnor fibres in any dimension do not contain exact Lagrangian tori.
\end{itemize}
Recall that the $A_k$-Milnor fibre is a plumbing of $k$ copies of $T^*S^n$ according to the chain graph.
The following was also proved in \cite{AS12}.

\begin{theorem}
	\label{th:gen_plum}
	Let $M$ be a plumbing of cotangent bundles of simply-connected manifolds according to a tree, and $\dim_\R M\ge 6$. Then the compact exact Fukaya category of $M$ is generated by the cores of the plumbing.\qed
\end{theorem}

The following corollary seems to be new.
	\begin{corollary}Assume that $k\le n$ and $n\ge 3$.
		The $n$-dimensional Milnor fibre $A_k$, i.e.~the plumbing of $k$ copies of $T^*S^n$ according to the linear graph, does not contain an exact Lagrangian torus with vanishing Maslov class.\qed
	\end{corollary}

\begin{proof}
	Let $X$ be an $n$-dimensional projective hypersurface of degree~$n$.  It clearly admits $A_k$-degenerations for $k\le n$, so the $n$-dimensional $A_k$ Milnor fibre $M$ embeds into $X$ (away from a hyperplane section which is an anticanonical divisor) when $k\le n$. (In fact, the actual bound on $k$ in terms of $n$ seems to be higher.) Hypothesis~\ref{hyp:const} holds for $X$ by Corollary~\ref{cor:hyp_comp_int}. So Theorem~\ref{th:gen_plum}, Corollary~\ref{cor:not_gen} and Corollary~\ref{cor:ve} imply the result.
\end{proof}

As another example, consider
the del Pezzo surface $X$ which is the blowup of $\C P^2$ at 3 points. It contains two Lagrangian spheres intersecting transversely once, disjoint from an anticanonical divisor. (If $E_1,E_2,E_3$ are the exceptional sphere classes, these Lagrangian spheres are in the classes $E_1-E_2$ and $E_2-E_3$, see e.g.~\cite{Ev10}.) Hence $X$ contains the $A_2$ Milnor fibre $M$ with the same property. It follows that $M$ does not contain an exact Lagrangian torus with vanishing Maslov class (which was known, as mentioned above). 

\subsection{Concluding notes}
Consider an affine variety $M$ which is the complement of an ample normal crossings divisor in a projective variety $X$, see e.g.~\cite{ML12}. Then $M$ is a Liouville domain, and after smoothing the given divisor, one obtains a Liouville embedding $M\subset X\setminus\Sigma$ where $\Sigma$ is smooth. The affine variety $M$ may admit many other compactifications; this is a broadly studied subject in algebraic geometry. Proposition~\ref{prop:no_exact} can be understood
as a set of conditions that prevent an affine variety $M$ from admitting a \emph{Fano} compactification.

Although constructing exact Lagragian tori in $M$ is complicated, see e.g.~\cite{Ke15}, one general approach to the construction is well known: one considers a Lefschetz fibration on $M$ and seeks to construct a Lagrangian torus as a matching path, reducing the problem to one inside the fibre.

We note that the applications of Proposition~\ref{prop:no_exact} provided above were only using the split-generation condition; we have not discussed the applications which would rely on the dimension of $SH^0(M)$ instead.
It may be interesting to look for such examples too, bearing in mind that $SH^0(M)$ is also related to Lefschetz fibrations on $M$ by an exact sequence involving the monodromy map, due to McLean \cite{McL10}.  

\section{Preparations}
\label{sec:prelim} 
In this section we set up the preparatory material for the proof of Theorem~\ref{th:factor}. We quickly remind the notions of superpotential, symplectic cohomology and the closed-open maps. Then we discuss positivity of intersections for Floer solutions and set down a class of \S-shaped Hamiltonians which will be used in the proof.

\subsection{The superpotential}
Let $X$ be a closed monotone symplectic manifold, and
 $L\subset X$ be a monotone Lagrangian submanifold.
Recall that $L$ is assumed to be oriented and spin. Denote $m=\rk H_1(L;\Z)/\mathrm{Torsion}$ and choose a basis of this group:
\begin{equation}
\label{eq:basis_H}
\Z^m\xrightarrow{\cong} H_1(L;\Z)/\mathrm{Torsion}.
\end{equation}
The potential of $L$ with respect to the basis (\ref{eq:basis_H}) is a Laurent polynomial
$$
W_L\co (\C^*)^m\to\C
$$
defined by:
\begin{equation}
\label{eq:potential_def}
W_L({\mathbf x})=\sum_{l\in\Z^m}\mathbf{x}^l\cdot\#\M_0^l
\end{equation}
where  $\mathbf{x}^l=x_1^{l_1}\ldots x_m^{l_m}$ and 
$\M_0^l$ is the moduli space of unparametrised $J$-holomorphic Maslov index 2 disks $(D,\bd D)\subset (X,L)$ passing through a specified point $\pt\in L$, and whose boundary homology class $[\bd D]$ equals $l\in \Z^m$ in the chosen basis (\ref{eq:basis_H}). The holomorphic disks are computed with respect to a regular tame almost complex structure $J$. Then $\M_0^l$ is 0-dimensional and oriented (as usual, the orientation depends on the spin structure), so the signed count $\# \M_0^l$ is an integer.

For an equivalent way of defining the superpotential, 
let $\rho$ be a {\it local system} on $L$, by which we mean (in a slightly non-standard way)
a map of Abelian groups 
$$\rho\co H_1(L;\Z)/\mathrm{Torsion}\to \C^*$$
where $\C^*$ is the multiplicative group of invertibles.
Using the basis (\ref{eq:basis_H}), one can view $\rho$ as a point: 
$$\rho\in (\C^*)^m,$$ 
by computing its values on the basis elements. We will use the two ways of looking at $\rho$ interchangeably.
Let 
$$\M_0=\bigsqcup_{l\in\Z^m}\M_0^l$$ be the
moduli space of all holomorphic Maslov index~2 disks as above, with any boundary homology class. Then one  puts
\begin{equation}
\label{eq:W_eval_at_rho}
W_L(\rho)=\#_{\rho}\M_0\in \C,
\end{equation}
where the right hand side denotes  the count of holomorphic disks in $\M_0$ {\it weighted using the local system}:
\begin{equation}
\label{eq:mu2_weighted_count}
\#_{\rho}\M_0=\sum_{l\in\Z^m}\# \M_0^l\cdot \rho(l)
\end{equation}  
Formula~(\ref{eq:W_eval_at_rho}) defines the value of the potential at any point of $(\C^*)^m$, and the resulting function is precisely the Laurent polynomial (\ref{eq:potential_def}), so the two definitions of the potential are consistent.

If one changes the basis (\ref{eq:basis_H}) by a matrix
$
(a_{ij})\in GL(m;\Z),
$
the corresponding superpotentials differ by a change of co-ordinates given by the multiplicative action of $GL(m;\Z)$ on $(\C^*)^m$:
\begin{equation}
\label{eq:change_coord}
\begin{array}{c}
x_i\mapsto x_i'=\prod_{j=1}^m x_j^{a_{ij}}, \quad \textit{ so that}\\
W_L'(x_1,\ldots,x_m)=W_L(x_1',\ldots, x_m').
\end{array}
\end{equation}

Recall that $GL(m;\Z)$ consists of integral matrices with determinant $\pm 1$. The proposition below is classical.

\begin{proposition}
	\label{prop:potential_invariance}
For a monotone Lagrangian submanifold $L\subset X$, its superpotential $W_L$, up to the change of co-ordinates (\ref{eq:change_coord}) corresponding to a change of basis (\ref{eq:basis_H}), is invariant under Hamiltonian isotopies of $L$, and more generally under symplectomorphisms of $X$ applied to $L$.\qed
\end{proposition}

\subsection{Symplectic cohomology}
\label{subsec:sh}
We assume that the reader is familiar with basic Floer theory, the definitions of symplectic cohomology, closed-open maps and related terminology, like~Liouville domains. The reader can consult e.g.~\cite{FH94, CFH95, Vi99, Sei08,CFO10,Ri13} for the necessary background.
Assuming familiarity with these notions, we give a quick overview in the amount required to set up the necessary notation.

Let $M$ be a Liouville domain with boundary $\bd M$. Its Liouville vector field gives a canonical parameterisation of the {\it collar} of $\bd M$ by 
$$[1-\delta,1]\times \bd M\subset M.$$
Here $\{1\}\times \bd M$ is the actual boundary of $M$.
In our definition of symplectic cohomology, we work directly with $M$ and not its Liouville completion; both ways are of course equivalent.

\begin{figure}[h]
	\includegraphics{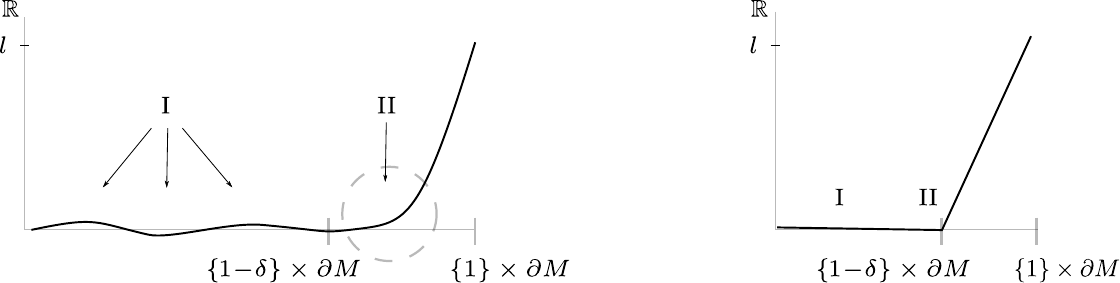}
	\caption{The Hamiltonians $\hat H_l$ used to define symplectic cohomology. The horizontal axis represents $M$, and particularly the collar co-ordinate near $\bd M$. Both pictures depict the same Hamiltonian, and we will use the right-hand `style' further.}
	\label{fig:h_l_sh}
\end{figure}

To define symplectic cohomology, we begin with class of Hamiltonians $M\to \R$ which are zero away from the collar, are monotone functions of the collar co-ordinate on the collar,
and which become linear of fixed slope in the collar co-rdinate starting from a certain distance to $\bd M$, say on $[1-\delta/2,1]\times \bd M$.
In the setup of symplectic cohomology, the important quantity to keep track of is the slope of the Hamiltonian near $\bd M$. For us it is more convenient to keep track of the maximum value of our Hamiltonians which is achieved on $\bd M$. We denote it by $l$ and call the \emph{height}, see Figure~\ref{fig:h_l_sh}. As $l$ tends to infinity so does the slope, which is good enough for the purpose of defining symplectic cohomology.

We perturb the Hamiltonians of the specified class by:
\begin{itemize}
	\item a perturbation away from the collar which turns the Hamiltonians into Morse functions away from the collar;
	\item optionally, a non-autonomous ($S^1$-dependent) perturbation in the collar which makes the 1-periodic orbits of the Hamiltonian flow non-negenerate.
\end{itemize}
We denote the resulting function after any such  perturbation by
\begin{equation}
\label{eq:h_sh}
\hat H_l\co S^1\times M\to\R,
\end{equation}
where $l$ is the maximum value (up to an $\epsilon$-error), or the height. All other choices, in particular the precise perturbations, are immaterial.
The shape of such function is sketched in Figure~\ref{fig:h_l_sh}~(left), and more crudely in Figure~\ref{fig:h_l_sh}~(right). We will adopt the crude variants of the pictures in the future.
The Floer complex $CF^*(\hat H_l)$  is generated, as a vector space over $\C$, by time-1 periodic orbits of the Hamiltonian vector field  $X_{\hat H_l}$. These orbits come in two types:

\begin{itemize}
	\item[\I:] Constant orbits that correspond to critical points of $\hat H_l$ away from the collar;
\item[\II:] Orbits in the collar that correspond to  Reeb orbits of $\bd M$ of various periods. 
\end{itemize}
By the maximum principle, solutions to Floer's equation never escape to $\bd M$, so there are well-defined Floer cohomology groups $HF^*(\hat H_l)$. We use the cohomological convention where positive punctures serve as inputs. This means that Floer's differential is given by: 
$$
d\gamma_+=\sum\# \M(\gamma_+,\gamma_-)\cdot \gamma_-
$$
where Floer solutions $u\in \M(\gamma_+,\gamma_-)$ have orbits $\gamma_{\pm }$ as their $s\to\pm \infty$ asymptotics, see Figure~\ref{fig:hf_dif}.

\begin{figure}[h]
	\includegraphics{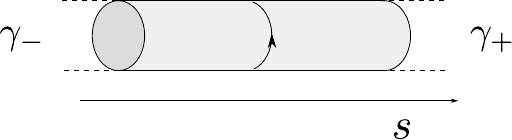}
	\caption{A curve for Floer's differential.}
	\label{fig:hf_dif}
\end{figure}

When $l\le l'$, it is easy to arrange that $\hat H_l\le \hat H_{l'}$ everywhere. If one sets up Floer's continuation equations using a homotopy $H(s)$ between $\hat H_{l'}$ and $\hat H_l$ ($l\le l'$) such that $\bd_s H\le 0$ everywhere, the solutions also obey the maximum principle. This means that there are well-defined continuation maps $HF^*(\hat H_l)\to HF^*(\hat H_{l'})$, $l\le l'$. The symplectic cohomology is the direct limit with respect to  these continuation maps:
\begin{equation}
\label{eq:def_sh}
SH^*(M)=\lim_{l\to+\infty} HF^*(\hat H_l).
\end{equation}
Symplectic cohomology acquires a $\Z$-grading once a trivialisation of the canonical bundle $K_M$ is fixed.

\subsection{Closed-open maps}
Let $M$ be a Liouville manifold and $L\subset M$ be an exact Lagrangian submanifold equipped with a local system $\rho$. Assume that $c_1(M)=0$ and fix a trivialisation of $K_M$. Next, assume that $L$  has vanishing Maslov class in $H^1(L)$ with respect to that trivialisation. Denote the pair $(L,\rho)=\LL$. The closed-open map is a map of graded algebras 
$$
\CO\co SH^*(M)\to HF^*(\LL,\LL).
$$
Here $HF^*(\LL,\LL)$ is  the self-Floer cohomology. Because $L$ is exact, there is an algebra isomorphism $$HF^*(\LL,\LL)\cong H^*(L;\C)$$
for any $\rho$. 
We denote by $1_L\in HF^*(\LL,\LL)$ the unit. We shall  write $HF^*_M(\LL,\LL)$ when we wish to emphasise that the Floer cohomology is computed inside $M$.

The definition of the closed-open map
goes by counting 
maps from the half-cylinder
$$
u\co [0,+\infty)\times S^1\to M
$$
solving the usual Floer's equation with Hamiltonian $\hat H_l$ (precisely the same one as used above), and with Lagrangian boundary condition $L$.
The closed-open map to $HF^*(\LL,\LL)=H^*(L)$ is then obtained by evaluating the solutions $u$ at the fixed boundary point $\{s=t=0\}$, and weighting the counts using the local system in the way it is done in (\ref{eq:mu2_weighted_count}).
Restricting to degree zero, one can write:
\begin{equation}
\label{eq:def_co}
\CO_l(\gamma)=\#_{\rho}\M(\pt,\gamma)\cdot 1_L\in HF^0(\LL,\LL)\quad\text{for }\quad \gamma\in CF^0(\hat H_l),
\end{equation}
where $\M(\pt,\gamma)$ counts Floer solutions  $u$ on the half-cylinder as above, asymptotic to  $\gamma$ as $s\to+\infty$ and sending the fixed boundary marked point $\{s=t=0\}$ to a specified point $\pt\in L$. 

Because the maps $\CO_l$ commute with the continuation maps, they define the closed-open map $\CO\co SH^*(M)\to HF^*(\LL,\LL)$ under the limit (\ref{eq:def_sh}). The definition of the closed-open map just given can be found in e.g.~\cite{Sei08}. The proof that it is a ring map follows the standard TQFT argument, see e.g.~\cite[Section~2]{She13}.

\subsection{Donaldson divisors}
A Donaldson divisor $\Sigma$ in a closed monotone symplectic manifold $X$ is
a smooth real codimension~2 symplectic submanifold whose homology class is dual to $dc_1(X)$, where $d$ is a positive intereger called the degree of $\Sigma$. Donaldson proved that such divisors always exist \cite{Do96}, and  the complement $X \setminus \Sigma$ has a natural Liouville structure, see e.g.~\cite[Section~4.1]{Pa13}.

 By \cite[Theorem~2]{AGM01}, \cite[Theorem~3.6]{CW17}, for a given monotone Lagrangian submanifold $L\subset X$, there exists a Donaldson divisor $\Sigma$ disjoint from it, and such that $L\subset X\setminus\Sigma$ is exact. In this setting,
 $\tfrac 1 d[\Sigma]=PD(\omega)$ is dual to twice the Maslov class of~ $L$. Namely,
 for a homology class $A\in H_2(X,L;\Z)$ one has:
 \begin{equation}
 \label{eq:mu_and_intersec}
 \mu(A)=\omega(A)=A\cdot[\Sigma]/2d
 \end{equation}
 where $\mu$ is the Maslov index and ``$\cdot$'' is the homology intersection, see e.g.~\cite[Lemma 3.4]{CW17}.

\subsection{Positivity of intersections with a Hamiltonian term}

\label{subsec:posit_intersec}
Let $\Sigma\subset X$ be a symplectic hypersurface, and fix a tame almost complex structure $J$ such that $T\Sigma$ is $J$-invariant; one says that $J$ preserves $\Sigma$.
It has been proved by Cieliebak and Mohnke \cite[Proposition~7.1]{CM07} (using the Carleman similarity principle from McDuff and Salamon \cite{MDSaBook}) that all intersections of $J$-holomorphic curves with $\Sigma$ are positive.
We will now discuss similar statements for solutions of Floer's equation, both at interior points and at  punctures.
We begin with intersections at interior points.

Denote by $U(\Sigma)$ a tubular neighbourhood of $\Sigma$.
Let $S\subset \C$
be a  domain, and suppose $u\co S\to U(\Sigma)$ solves Floer's equation:
\begin{equation}
\label{eq:Floer_Hst}
\bd_su+J\bd_tu=JX_{H(s,t)}.
\end{equation}
Here
$s={\mathit \Re}\, z$, $t={\mathit \Im}\, z$ for the complex  co-ordinate $z\in\C$, and one takes a domain-dependent Hamiltonian $H\co S\times X\to\R$. One proves positivity of intersections using Gromov's trick, which reduces the question to the Cieliebak-Mohnke setting. The lemma below appeared in e.g.~\cite[Lemma~4.3]{GP16}.

\begin{lemma}[Positivity of intersections]
\label{lem:pos_inter}
Suppose that $J$ preserves $\Sigma$, and $X_{H(s,t)}|_\Sigma$ is tangent to $\Sigma$ for all $s,t$.
Let $u\co S\to U(\Sigma)$ be a solution of (\ref{eq:Floer_Hst}) such that  $u(S)\cap\Sigma\neq\emptyset$ and $u(\bd S)\cap\Sigma=\emptyset$, then $u(S)$ has positive intersection number with $\Sigma$.\qed
\end{lemma}

By the intersection number, we mean  the following.
Given that $u(\bd S)$ is disjoint from $\Sigma$, it produces a well-defined  class $[u(S)]\in H_2(U(\Sigma),\bd U(\Sigma);\Z)$.
We consider the intersection number between $[u(S)]$ and $[\Sigma]\in H_{2n-2}(U(\Sigma);\Z)$.

Maps from cylinders of finite energy solving Floer's equation converge to periodic Hamiltonian orbits (at least when the orbits are non-degenerate).
If the orbit in question is a constant orbit in $\Sigma$, it makes sense to speak of the intersection sign of the compactified curve with $\Sigma$ at that constant orbit. 
The occurring intersection numbers have recently been analysed by Seidel \cite{Sei16}, and we will use a slight reformulation of his result.

Let $S$ be one of the two domains:
$$
S_-=(-\infty,R]\times S^1,\quad S_+=[R,+\infty)\times S^1,
$$
where $R\in\R$.
We use the co-ordinates $(s,t)\in S$ where $t\in S^1$ and $s\in (-\infty,R]$ or $[R,+\infty)$.
This time, we are interested in Floer's equation with autonomous Hamiltonian:
\begin{equation}
\label{eq:Floer_H}
\bd_su+J\bd_tu=JX_{H},
\end{equation}	
where $H\co U(\Sigma)\to \R$. Since we are only interested in the behaviour of Floer solutions near $\Sigma$, it suffices to have $H$ defined on $U(\Sigma)$.

\begin{lemma}[Intersection at asymptotics {\cite[Equation~(7.22)]{Sei16}}]
\label{lem:pos_inter_asympt}
Let $p\in \Sigma$ be a Morse critical point of $H$, and assume that there is a chart for $X$ at $p$ mapping a neighbourhood of $p$ to a neighbourhood of the origin in $\R^{2n-2}\times \R^2$ with the following properties. It takes $\Sigma$ to $\R^{2n-2}\times \{0\}$, $\omega$ to the standard form, $J$ to a split complex structure of the form $\left(\begin{smallmatrix}
J_{\Sigma} & 0\\ 0&  i
\end{smallmatrix}\right)$, where $i$ is the standard complex structure on $\R^2$ but $J_{\Sigma}$ may vary with the point on $\R^{2n-2}$. Finally, we require that $H$ is taken to the following form: $$H=H_{\Sigma}(x_1,\ldots,x_{2n-2})-\pi \alpha(x_{2n-1}^2+x_{2n}^2)$$
for some $\alpha\in \R\setminus\Z$. 

Let $S$ be one of the two domains: $S_+$ or $S_-$, and 
$u\co S\to U(\Sigma)$ be a solution of (\ref{eq:Floer_H}) asymptotic to $p$ as $s\to\pm\infty$, seen as a constant periodic orbit. Assume that   $u(\bd S)\cap\Sigma=\emptyset$, then
the intersection number $[\overline{u(S)}]\cdot [\Sigma]$ is greater than or equal to:
\begin{itemize}
\item[(i)] $\lfloor \alpha \rfloor+1$ if $S=S_+$, or
\item[(ii)] $-\lfloor \alpha \rfloor$  if $S=S_-$. \qed
\end{itemize}
\end{lemma}
In the above lemma, $\overline{u(S)}=u(S)\cup\{p\}$ is the closure of $u(S)$. Then $p$ is, by hypothesis, an intersection point between $\overline{u(S)}$ and $\Sigma$. The proof of the lemma is quite straightforward, because the splitting assumption makes Floer's equation completely standard in the normal $\R^2$-direction responsible for the intersection multiplicity, and that equation can be explicitly solved.

\begin{remark}
To get familiar with the lemma, it is helpful to consider the case of flowlines. A $t$-independent Floer solution $u\co S_{\pm \infty}\to X$ is a gradient flowline of $H$, flowing \emph{down} in the direction $s\to +\infty$. Suppose $\alpha>0$, then $H$ is decreasing in the direction normal to $\Sigma$. Hence there exist downward gradient flowlines flowing away from $\Sigma$ (i.e.~asymptotic to $\Sigma$ at the negative end), but no flowlines flowing towards $\Sigma$ (i.e. ~asymptotic to $\Sigma$ at the positive end). The former flowlines, which exist, obviously have intersection number $0$ with $\Sigma$, showing that the bound in Lemma~\ref{lem:pos_inter_asympt}(ii) is optimal in this case. 
\end{remark}

Finally, let us explain how to find $J$ and $H$ that satisfy the conditions of Lemmas~\ref{lem:pos_inter} and~\ref{lem:pos_inter_asympt}. 
Let $\Sigma\subset (X,\omega)$ be a Donaldson hypersurface given as the vanishing set of an almost holomorphic section of a complex line bundle on $X$. Let $\cL\to \Sigma$ be the restriction of that line bundle to $\Sigma$.  Fix a Hermitian metric on $\cL$, and a connection $\nabla$ with curvature $-2\pi i d\cdot \omega|_\Sigma$, where $d$ is the degree of $\Sigma$. The total space $\cL$ carries a canonical symplectic form considered e.g.~by Biran \cite{Bi01}, given by
\begin{equation}
\label{eq:omega_on_line_bdle}
\omega_0=\pi^*\omega|_{\Sigma}+d(r^2\alpha^\nabla)
\end{equation}
where $\pi\co\cL\to\Sigma$ is the projection, $\alpha^\nabla$ is the circular 1-form associated with the connection $\nabla$, and $r$ is the fibrewise norm. Let $E_\cL$ be the radius-$\epsilon$ neighbourhood of the zero-section of $\cL$, for some small fixed $\epsilon$.
By the symplectic neighbourhood theorem \cite[Section~2]{Bi01}, there is a neighbourhood  $U(\Sigma)\subset X$ which is symplectomorphic to $(E_\cL,\omega_0)$.  

Fix a function $H_\Sigma\co \Sigma\to \R$.
Consider the autonomous Hamiltonian
\begin{equation}
\label{eq:h_on_bdle}
H=l+H_\Sigma\circ \pi-\alpha r^2\co E_\cL\to\R,
\end{equation}
where $l\in \R$ is some constant  and $\pi\co E_\cL\to \Sigma$ is the projection.
Let $J$ be the almost complex structure which is the $\nabla$-lift of an almost complex structure on $\Sigma$ to $E_\cL$ (preserving the horizontal distribution and the fibres); then $J$ preserves $\Sigma$. Our standing setup will be to use the parameter
$$
0<\alpha<1,\quad{equivalently,}\quad \lfloor\alpha \rfloor= 0.
$$

The chosen $J$, $H$ satisfy the conditions of Lemma~\ref{lem:pos_inter} but not Lemma~\ref{lem:pos_inter_asympt}:
since $\nabla$ has non-zero curvature, it is not possible to bring all of $J,H,\omega$ to a split form specified in Lemma~\ref{lem:pos_inter_asympt} in a neighbourhood of a point. However, one may homotope $\nabla$ to a flat connection over a neighbourhood  of each critical point of $H_\Sigma$, assuming this finite set of points has been chosen in advance. Then (\ref{eq:omega_on_line_bdle}) gives another symplectic form on $E_\cL$ which is diffeomorphic to the standard one, by Moser's lemma. After the connection is modified this way, the above choice of $H,J$ will satisfy the conditions of Lemmas~\ref{lem:pos_inter},~\ref{lem:pos_inter_asympt}. We remind that the constant $\alpha$ affects the conclusion of Lemma~\ref{lem:pos_inter_asympt}; and that we are using $0<\alpha<1$.

\begin{remark}
It seems plausible that Lemma~\ref{lem:pos_inter_asympt} is true for all Hamiltonians of the form (\ref{eq:h_on_bdle}) without assuming that $\nabla$ is flat near the critical points.
\end{remark}

\begin{remark}
In Seidel's setup of \cite{Sei16}, the normal bundle to $\Sigma$ is trivial. Rephrasing \cite[Setup 7.2]{Sei16} in our language, he chooses  the globally flat $\nabla$ which ensures the splitting as in Lemma~\ref{lem:pos_inter_asympt} near all points of $\Sigma$.
\end{remark}

\subsection{\S-shaped Hamiltonians}
\label{subsec:s_shape}
We are going to introduce the main class of Hamiltonians that will be used in the proof of Theorem~\ref{th:factor}.
Let $X$ be a symplectic manifold and $S\subset X$ a  contact type hypersurface. A small neighbourhood $U(S)$ of $S$ admits a Liouville vector field whose flow identifies $U(S)$ with $[1-\delta,1+\delta ]\times S$. We call $U(S)$ equipped with such an identification a {\it Liouville collar}, and the parameter $r\in [1-\delta,1+\delta]$ the {\it radial co-ordinate} on the Liouville collar. The original hypersurface $S$ is identified with the middle of the collar: $\{1\}\times S$.

\begin{remark}
\label{rem:hypersurf_liouv}
When proving Theorem~\ref{th:factor}, the specific setup will be to take $S=\bd M$ the boundary of a Liouville domain inside $X$. See Figure~\ref{fig:s_shape} (left). 
\end{remark}

\begin{figure}[h]
\includegraphics{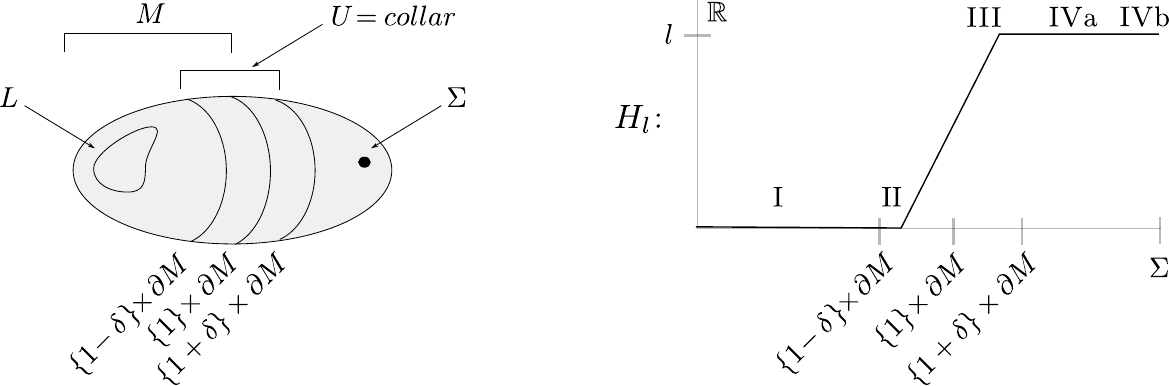}
\caption{Hamiltonians $H_l$ of this shape will be used to prove Theorem~\ref{th:factor}. Here $S=\bd M$.}
\label{fig:s_shape}
\end{figure}

Let $X_-\sqcup X_+=X\setminus U(S)$ be the two components of the complement of $U(S)$; with $X_-$ being the component with convex boundary $\{1-\delta\}\times S$.
Take $l\in\R_+$ and
define the unperturbed \S-shaped Hamiltonian $H_{l}^0\co X\to \R$ by:
\begin{equation}
\label{eq:H_l_0}
\begin{array}{ll}
H_{l}^0\equiv 0 & \text{ on } X_-, \\
H_{l}^0(x)=h_l(r(x)) & \text{ for } x\in U(S),\\
H_{l}^0\equiv l & \text{ on } X_+, 
\end{array}
\end{equation} 
Here $r(x)\in[1-\delta,1+\delta]$ is the radial co-ordinate of a point in the collar, and $h_l(r)\co [1-\delta,1+\delta]\to\R$ is a function of the collar co-ordinate whose shape is shown in Figure~\ref{fig:s_shape}~(right), taking $S=\bd M$. The parameter $l$ is its maximum value. Note that $\delta$ is another parameter of the construction, but it will be fixed throughout so we do not include it in the notation. 
Finally, let 
\begin{equation}
\label{eq:H_l_S}
H_{l}\co X\times S^1\to \R
\end{equation} 
be a (optionally, $t$-dependent) perturbation of $H^0_{l}$ with the following properties: 
\begin{enumerate}
\item[(i)] the perturbation turns $H_{l}$ into a Morse function on $X_-$ and $X_+$, and we assume for convenience that $H_l$ is $t$-independent in those regions;

\item[(ii)] near $\Sigma$, $H_{l}$ has the form (\ref{eq:h_on_bdle}) described in Subsection~\ref{subsec:posit_intersec} after flattening the connection near its critical points on $\Sigma$ as described there. Then $H_l$ satisfies Lemmas~\ref{lem:pos_inter} and~\ref{lem:pos_inter_asympt} for some tame $J$ preserving $\Sigma$. We shall only be using almost complex structures with this property, which are additionally cylindrical in the collar $[1-\delta,1+\delta]\times S$. We also require that the function $H_\Sigma$ from (\ref{eq:h_on_bdle}) is $C^2$-small;
\item[(iii)] optionally, we may use a $t$-dependent perturbation ($t\in S^1$) in the collar region. If we choose to do so, we perform the perturbation only in the subregion $[1-\delta,1]\times S$ of the collar $U(S)\subset X$; elsewhere $H_{l}$ is independent of $t$.
\end{enumerate}
We call $l$ the \emph{height} of $H_l$.
We use Figure~\ref{fig:s_shape} to depict this pertured Hamiltonian as well. The 1-periodic orbits of $X_{H_l}$ are divided into the following types, using the standard notation (see e.g.~\cite{CO18}):

\begin{enumerate}
\item[\I:] constant orbits in $X_-$;
\item[\II:] orbits in $[1-\delta,1]\times S$ arising in the region where $h_l''>0$, corresponding to Reeb orbits in $S$.
\item[\III:] orbits in $[1,1+\delta]\times S$ arising in the region where $h_l''<0$, corresponding to Reeb orbits in $S$.
\item[\IV a:] constant orbits in $X_+$ lying away from $\Sigma$;
\item[\IV b:] constant orbits in $X_+$ lying in $\Sigma$. Such orbits always exist by our choice of the Hamiltonian, see (\ref{eq:h_on_bdle}).
\end{enumerate}

Note that the constant orbits \I, \IV a, and \IV b are Morse, and therefore non-degenerate.

A comment about time-dependent versus autonomous Hamiltonians is due. Depending on whether one uses the optional time-dependent perturbation appearing above, the type \II\ orbits will either be fully non-degenerate (two perturbed Hamiltonian orbits corresponding to a Reeb orbit), or will come in $S^1$-families, where the $S^1$ action is the rotation in $t$. Although this is not crucial, we always choose type \III\ orbits to be unperturbed and come in $S^1$-families. In Section~\ref{sec:proofs} we shall focus on the setting where the type \II\ orbits are perturbed, but  the arguments work equally well in the autonomous setting, after the usual adjustments following the $S^1$-Morse-Bott framework for symplectic cohomology of Bourgeois and Oancea \cite{BO09}. Later in Section~\ref{sec:antican}, we will need the autonomous framework to perform a computation, and we will mention the relevant adjustments therein.

Let us now record a special case of a lemma due to Bourgeois and  Oancea~\cite[pp.~654-655]{BO09b}; see
\cite[Lemma~2.3]{CO18} for a detailed discussion of it.

\begin{lemma}[Asymptotic behaviour]
\label{lem:prohibited_trj}
Let $S_+$ be the domain $[R,+\infty)\times S^1$ and $u\co S_+\to X$ a solution to Floer's equation (\ref{eq:Floer_Hst}) with the Hamiltonian $H_{l}$ as above. Assume that $u$ is asymptotic, as $s\to+\infty$, to a type \III\ orbit $\gamma$ of $H$ lying in $\{1+r_0\}\times S$, for some $0<r_0<\delta$. Then there exists $s>0$ such that 
$u(s,t)\in (1+r_0,1+\delta]\times S$. \qed
\end{lemma}

\begin{remark}
We do not want to perturb $H_{l}$ in a $t$-dependent way in the region $[1,1+\delta]\times S$ precisely to allow us to refer to \cite{BO09} in a most straighforward way, as this reference considers autonomous Hamiltonians. However, the same result also holds for small $t$-dependent perturbations of the Hamiltonian by the Floer-Gromov compactness of \cite{BO09}. We actually learned Lemma~\ref{lem:prohibited_trj} from \cite{CO18} which applies it in the perturbed context.
\end{remark}

\section{Proof}
\label{sec:proofs}

This section contains the proofs of Theorems~\ref{th:factor} and Theorem~\ref{th:factor_higher} as well as the definition of the Borman-Sheridan class and higher deformation classes $\D_k$.

\subsection{Stabilising divisor} 
Our first step is quite standard; it is inspired by the Auroux-Kontsevich-Seidel lemma \cite[Section 6]{Au07} and the technique of stabilising divisors due to Cieliebak and Mohnke \cite{CM07}.
Fix a tame almost complex structure $J$ such that $\Sigma$ is a $J$-complex hypersurface. 
Consider two disk counting problems. The original count we are interested in comes from~(\ref{eq:mu2_weighted_count}):
$$
m_0=\#_{\rho(L)} \M_0=W_L(\rho).
$$
Recall that $\M_0$ consists of unparametrised holomorphic Maslov index 2 disks with boundary on $L$, and  passing through a fixed point $\pt\in L$.
One introduces another number: 
$$m=\#_{\rho} \M$$
counting parametrised holomorphic Maslov index 2 maps $u\co(D,\bd D)\to (X,L)$ such that $u(1)=\pt$ is a fixed point on $L$, and $u(0)\in \Sigma$. The count is again weighted using $\rho$. Denote by $\M$ the 0-dimensional moduli space just introduced. Here $D$ is the unit disk in $\C$, $1\in D$ is the fixed point on its boundary, and $0\in D$ is the fixed point in the interior. Our first claim is that
\begin{equation}
\label{eq:k_to_1}
m=dm_0.
\end{equation}
Recall that the natural orientations (signs) on $\M_0$ and $\M$ arise from regarding them as fibre products:
$$
\M_0=(\M_{1,0}(2))\times_{\ev_1}(\pt),\quad 
\M=(\M_{1,1}(2))\times_{\ev_1\times \ev_2}(\pt\times \Sigma)
$$
where $\M_{1,0}(2)$ resp.~$\M_{1,1}(2)$ are the moduli spaces of Maslov index 2 disks with one boundary marked point, resp.~one boundary and one interior marked point; and $\ev_1$, $\ev_2$ is the evaluation at the boundary resp.~interior marked point.
By Equation~(\ref{eq:mu_and_intersec}), the image $u(D)$, $u\in \M_0$, has algebraic intersection number $d$ with  $\Sigma$. If one marks the preimage of any such intersection point $p$ and finds the unique reparametrisation $\tilde u$ of $u$ such that $\tilde u(1)=\pt$, $\tilde u(0)=p$, one obtains an element $\tilde u\in \M$ whose orientation sign coincides with the intersection sign. Therefore the converse forgetful map $\M\to\M_0$ has degree $d$; Equation~(\ref{eq:k_to_1}) follows. For the rest of the proof, we shall work with $\M$ instead of $\M_0$.

\subsection{Domain-stretching.}
\label{subsec:stretch}
Let $M\subset X\setminus \Sigma$ be a Liouville subdomain.
Fix the height parameter $l\in\R_+$. Let 
$$H_l\co X\times S^1\to \R$$ 
be an \S-shaped Hamiltonian of height $l$ described in Subsection~\ref{subsec:s_shape}, see (\ref{eq:H_l_0}), (\ref{eq:H_l_S}), where  $\bd M$ is used as the contact type hypersurface $S$.

Our goal is to introduce, for each fixed $l$, a sequence of domain-dependent Hamiltonian perturbations of the holomorphic equation for the Maslov index 2 disks appearing in the definition of $\M$, by inserting the Hamiltonian $H_l$ as a perturbation over a sequence of annuli exhausting the unit disk $D$.  The  sequence will be parametrised by the integers $n\to\infty$. Introducing this sequence of perturbations can be called {\it domain-stretching}, by analogy with  neck-stretching in the SFT sense. It is important to keep in mind the following contrast with SFT neck-stretching: while the latter procedure changes the holomorphic equation by modifying the almost complex structure in the \emph{target} $X$, we modify the equations over the \emph{domain}, the unit disk. Domain-stretching is part of the standard Floer-theoretic toolbox; for example, it is used to prove composition rules for continuation maps in Floer cohomology (in which case the domain is the cylinder rather than the disk). 

It is helpful to  consider, instead of the unit disk $D$, a differently parametrised disk 
$$D_n:=A_n\cup B$$ where 
$$A_n=S^1\times[0,n]$$
is an annulus and $B$ is a unit disk capping the annulus $A_n$ from the right side, see Figure~\ref{fig:abc}. On $A_n$, we introduce the following standard co-ordinates: $t\in S^1$ and $s\in[0,n]$. We extend them to polar
co-ordinates on $B$ using the vector fields
$\bd_s$, $\bd_t$ shown in Figure~\ref{fig:abc}. Note that $\bd_s$, $\bd_t$ vanish at one point on $B$. Denote this point by $0\in B$; also, denote the point $\{s=t=0\}$ in $A_n$ by $1\in A_n$ (this notation is unusual, but it is consistent with the boundary point $1\in D$ used in the definition of $\M$ above).

Consider the complex structure on the disk $D_n$ glued from the standard complex structure on $A_n$ (which depends on $n$) and the standard complex structure on $B$ (which does not). The resulting complex structure on $D_n$ is of course biholomorphic to the unit disk $D$, but we will use the presentation $D_n=A_n\cup B$ to introduce domain-dependent perturbations.

We now define a sequence of domain-dependent Hamiltonians
$$
H_{n,l}\co D_n\times X\to\R.
$$ 
We introduce them separately over each of the two pieces $A_n$ and $B$; this means that we specify the restrictions $H|_{A_n\times X}$ and $H|_{B\times X}$.

\begin{enumerate}
\item[$A_n:$] Set  
$$
H_{n,l}(t,s,x)= H_l(t,x)\co   S^1\times [0,n]\times X\to \R.
$$
Here $H_l$ is the \S-shaped Hamiltonian of height $l$ introduced in Subsection~\ref{subsec:s_shape} where we use the parameter $0<\alpha<1$.
Observe that $H_{n,l}$ is an $s$-independent Hamiltonian in this region.
\item[$B$:] 
Set $H_{n,l}\equiv 0$ near the point $0\in B$ (recall this is the point where $\bd_s,\bd_t$ vanish) and $H_{n,l}\equiv H_l$ near $\bd B$. Over the sub-annulus of $B$ shown in light shade in Figure~\ref{fig:abc}, $H_{n,l}$ interpolates between $H_l$ and zero in such a way that $\bd_s H_{n,l}\le 0$ everywhere. 

We make the following additional provisions. We require that $X_{H_{n,l}(s,t,x)}|_\Sigma$ is small and is always tangent to $\Sigma$. 
For convenience we also assume that as one moves from left to right over the interpolation region (the light-shaded annulus), $H_{n,l}$ is first given near $\Sigma$ by varying the constant $l$ from (\ref{eq:h_on_bdle}) from the given height parameter $l$ to a small constant $\epsilon$ keeping the rest of the data in (\ref{eq:h_on_bdle}) fixed, and is subsequently homotoped monotonically to zero.
\end{enumerate}

Figure~\ref{fig:abc} gives an impression of how the function $H_{n,l}$  looks. It shows a movie of functions $H_{n,l}(s,t,\cdot)\co X\to \R$ for various values of $s$ (the dependence on $t$ may be small, and does not matter). 

\begin{figure}[h]
	\includegraphics{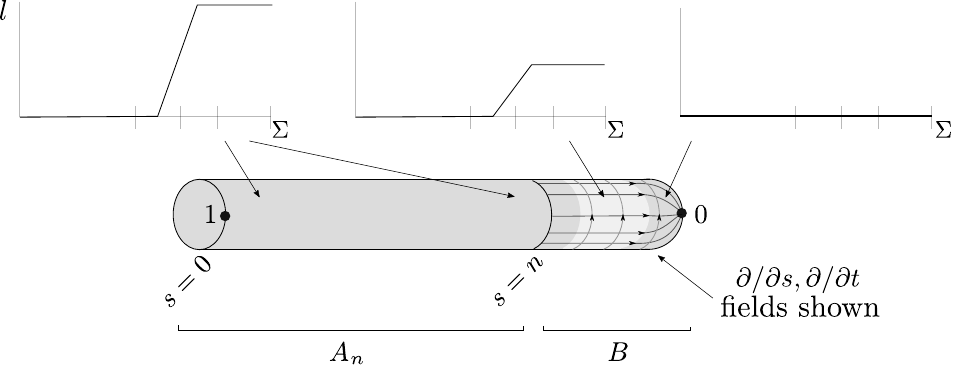}
	\caption{Domain-dependent Hamiltonians $H_{n,l}$ used in the domain-stretching sequence.}
	\label{fig:abc}
\end{figure}

We make a choice of $H_{n,l}$ for each $n,l$ and arrange that  $H_{n,l}\le H_{n,l'}$ for $l\le l'$,  pointwise on $D_n\times X$.
In particular, over the region $A_n$ we are using 
\S-shaped Hamiltonians from Subsection~\ref{subsec:s_shape} satisfying $H_l\le H_{l'}$. For now, we continue to work with a fixed $l$.

One identifies $D_n$ conformally with the unit disk $D$, so that $(s,t)$ become the standard polar co-ordinates on $D$, and the respective points $0,1$ in $D_n$ and $D$ are identified.
This uniformisation allows to treat our Hamiltonians as being defined on the same domain, $H_{n,l}\co D\times X\to \R$.

Consider the following perturbed holomorphic equation for $u\co (D,\bd D)\to (X,L)$:
\begin{equation}
\label{eq:perturbed_holoc_disk}
\bd_su+J\bd_t u=JX_{H_{n,l}}
\end{equation} 
where $X_{H_{n,l}}$ is the Hamiltonian vector field with respect to the variable in $X$.
The vector fields $\bd_s,\bd_t$ vanish at the origin $0\in D_n=D$, so  Equation~(\ref{eq:perturbed_holoc_disk}) does not make sense as written over that point. However, because $H_{n,l}\equiv 0$ near the origin, the equation simply restricts to the unperturbed $J$-holomorphic equation near the origin, so extends over it.

In what follows, we use an almost complex structure which coincides near $\Sigma$ with the one used in Subsection~\ref{subsec:posit_intersec}, is contact type on the collar region, and is generic otherwise.

Let $\M_{n,l}$ be the moduli space of Maslov index 2 maps solving Equation~(\ref{eq:perturbed_holoc_disk}) which additionally satisfy $u(1)=\pt\in L$, $u(0)\in\Sigma$, where $\pt\in L$ is a fixed point. These spaces are 0-dimensional and regular   (there is obviously enough freedom to ensure transversality; for example, by perturbing $J$ near $L$). 

A homotopy from $X_{H_{n,l}}$ to zero in Equation~(\ref{eq:perturbed_holoc_disk}) proves that for each $n,l$:
\begin{equation}
\label{eq:count_m_n_l}
\#_{\rho}\M_{n,l}=\#_{\rho}\M=m.
\end{equation}
(Disk or sphere bubbling is excluded because $L$ is monotone and the disks have Maslov index 2, and therefore have the lowest possible symplectic area.)

\subsection{Breaking and gluing}
Recall that the Hamiltonian $H_{n,l}$ restricts to the $s$-independent Hamiltonian $H_l$ in the region $A_n\subset D_n$ of the domain.
This is a setting to which the standard Floer-Gromov compactness theorem \cite{Flo89} applies.  It states that solutions in $\M_{n,l}$ converge, as $l$ is fixed and $n\to\infty$, to {\it broken curves} like the one shown in Figure~\ref{fig:a_b}. The breakings happen along 1-periodic Hamiltonian orbits of $H_l$.

\begin{figure}[h]
\includegraphics{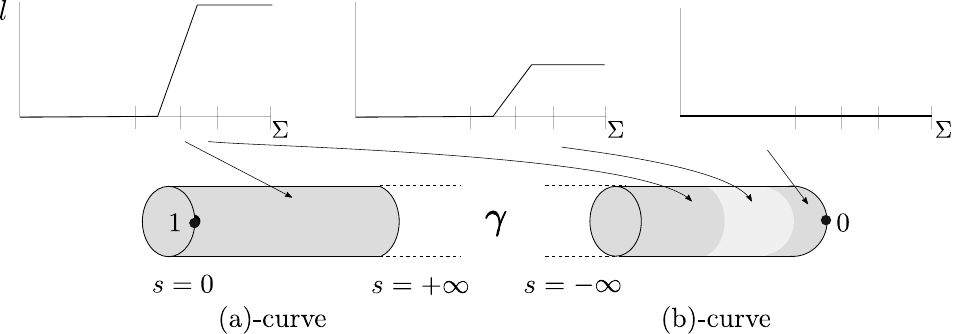}
\caption{A broken configuration consists of the (a)- and the (b)-curve, which carry the depicted Hamiltonian perturbations $H_l^{(a)}$, $H_l^{(b)}$.}
\label{fig:a_b}
\end{figure}

The  broken curve contains two main parts (see Figure~\ref{fig:a_b}): we call them the (a)-curve and the (b)-curve. In addition, the broken curve may contain: cylinders solving Floer's differential equation with respect to $H_l$ inserted between the (a)- and (b)-curves, and disk or sphere bubbles, as shown in Figure~\ref{fig:gen_breaking}. Some sphere bubbles may be contained inside $\Sigma$, because $J$ preserves $\Sigma$. The claim is that these extra parts cannot appear.
 
 \begin{figure}[h]
 	\includegraphics[]{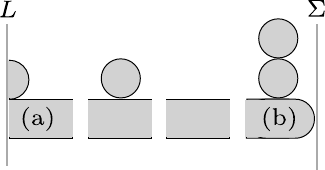}
 	\caption{A general Floer-Gromov limit.}
 	\label{fig:gen_breaking}
 \end{figure}
 
 \begin{lemma}
 	\label{lem:a_b}
When $l$ is large enough (greater than the area of a Maslov index~2 disk with boundary on $L$), any broken curve which is the Floer-Gromov limit of curves in $\M_{n,l}$ is composed of only two parts: the (a)- and the (b)-curve.
 \end{lemma}
 
We shall give a proof of Lemma~\ref{lem:a_b}  later in this section, because the proof will re-use arguments appearing in the next subsection. So we proceed assuming Lemma~\ref{lem:a_b}.

Let us discuss the (a)- and the (b)-curves in more detail.
The domain of the (a)-curve is $[0,+\infty)\times S^1$. 
The domain of the (b)-curve is bi-holomorphic to $\C$, but it is more convenient to look at it as on $$((-\infty,0]\times S^1)\cup B,$$ where $B$ is the capping disk appearing earlier.
The (a)- and (b)-curves solve Floer's equation
(\ref{eq:Floer_Hst}) with respect to the Hamiltonians
$$
H^{(a)}_l\co [0,+\infty)\times S^1\times X\to \R \quad\textit{and}\quad
H^{(b)}_l\co \left((-\infty,0]\times S^1)\cup B\right)\times X\to \R.
$$
given by:
\begin{equation}
\label{eq:H_a_and_b}
\begin{array}{l}
H^{(a)}_l(s,t,x)\equiv H_l(t,x)\quad\text {(the $s$-independent \S-shaped Hamiltonian)} ;\\
H^{(b)}_l\text{ is stitched using } H_l \text{ over }(-\infty,0]\times S^1
,\text{ and }  H_{n,l}|_B\text { over }B.
\end{array}
\end{equation}

Both the (a)- and the (b)-curve must be asymptotic to the same periodic orbit which is denoted by: 
$$\gamma,\text{ a 1-periodic orbit of } H_l.$$ 
(For type \II\ and \III\ orbits, whenever they are autonomous, the (a)- and (b)-curve may converge to different parametrisations of $\gamma$).
The incidence conditions $u(1)=\pt\in L$ and $u(0)\in \Sigma$ must be met by the new curves; now the marked point $0$ belongs to the (b)-curve and the marked point $1$ belongs to the (a)-curve as shown in Figure~\ref{fig:a_b}.

\begin{proposition}
\label{prop:gluing_broken}
For each sufficiently large $l$, the count of the configurations consisting of an (a)-curve and a (b)-curve  sharing an asymptotic orbit $\gamma$, equals 
$$ d\cdot  W_L(\rho),$$
i.e.~the count of the (a)-curves is weighted using the local system $\rho$.
\end{proposition}

\begin{proof}
By Gromov-Floer compactness and Floer's gluing theorems, the count of such configurations equals $\#_{\rho} \M_{n,l}$ for large enough $n$. Then the desired equality follows from (\ref{eq:count_m_n_l}) and (\ref{eq:k_to_1}).
In the $S^1$-Morse-Bott case,  the gluing theorem is due to \cite{BO09}. 
\end{proof}

\subsection{Ruling out type \IV b orbits}
Recall the discussion of the periodic orbits of the \S-shaped Hamiltonian $H_{l}$ in Subsection~\ref{subsec:s_shape}: they come in groups \I,\ \II,\ \III,\ \IV a,\ \IV b. Our aim is to show that an orbit $\gamma$ arising from a broken curve as above must necessarily be a type \I\ or \II\ orbit when $l$ is large enough.
Let us reformulate the hypothesis that the embedding $M\subset X\setminus\Sigma$ is grading-admissible in the following convenient way.

\begin{lemma}
\label{lem:div_by_k}
The image of the intersection pairing $-\cdot  [\Sigma]\co H_2(X,M;\Z)\to \Z$ is contained in $d\Z$.
\end{lemma}

\begin{proof}
	Let $\eta$ be the natural trivialisation of $(K_{X\setminus \Sigma})^d$.
	The obstruction $o$ to finding a $d$th root of $\eta$ over $M$ lies in $H^1(M;\Z/d\Z)$, see e.g.~\cite[Section~2.2]{PT17}. On the image of $H_2(X,M)\to H_1(M)$, this obstruction can be computed as follows. Consider a 2-chain $B$ in $(X,M)$. 
	By definition of the natural trivialisation, $\eta|_B$ is given by a section $s^{-1}$ of $(K_X)^d|_B$ having pole of order $1$ at each intersection point $B\cap \Sigma$. It follows that the value of the obstruction $o\cdot [\bd B]\in\Z/d\Z$ equals  the intersection number $B\cdot[\Sigma]$ mod~$d$. These pairings are all zero if and only $o$ vanishes, which implies the lemma.
\end{proof}

In what follows, by a \emph{limiting  orbit $\gamma$} we mean a periodic orbit of $H_l$ arising as the common asymptotic of some broken configuration of an (a)- and a (b)-curve which is the Floer-Gromov limit of a sequence of solutions in $\M_{n,l}$ as $n\to+\infty$. Assuming Lemma~\ref{lem:a_b}, we prove the following.

\begin{lemma}
\label{lem:not_4b} 
When $l$ is greater than the area of a Maslov index~2 disk with boundary on $L$, a limiting orbit $\gamma$ cannot be of type \IV b, i.e.~a constant orbit in $\Sigma$. 
\end{lemma}

\begin{proof}
Because type \IV b orbits are constant, the images of the (a)- and the (b)-curve can be compactified by adding the corresponding asymptotic point. These two curves therefore define homology classes which we call $A\in H_2(X,L)$ and $B\in H_2(X)$, for the (a)- and (b)-curve respectively. (We use coefficients in $\Z$.) The (b)-curve may be entirely contained in $\Sigma$, because $J$ preserves $\Sigma$ and the Hamiltonian vector field of $H_l^{(b)}$ is always tangent to $\Sigma$. Consider the two cases separately; see Figure~\ref{fig:ivb}.

\begin{itemize}
	\item If the (b)-curve is contained in $\Sigma$, then the Hamiltonian perturbation in its equation is small, as the function $H_\Sigma$ from (\ref{eq:h_on_bdle}) is taken to be small. This means that the image of the (b)-curve is close to being $J$-holomorphic; it follows that $\omega(B)\ge 0$ (a precise argument appears in Lemma~\ref{lem:pos_area} below). Therefore $c_1(B)\ge 0$ by monotonicity, where the Chern class is computed inside $X$, not $\Sigma$. Observe that the (b)-curve may be constant, in which case $B=0$. It follows that $$B\cdot[\Sigma]\ge 0,$$ since $\Sigma$ is dual to a positive multiple of $c_1(X)$.

	\item If the (b)-curve is not contained inside $\Sigma$, 
	note that it intersects $\Sigma$ at least once, by the incidence condition $u(0)\in\Sigma$. Due to positivity of intersections, using both Lemmas~\ref{lem:pos_inter} and~\ref{lem:pos_inter_asympt} with $\lfloor \alpha\rfloor= 0$, one obtains that $$B\cdot[\Sigma]\ge 1.$$
\end{itemize}	

\begin{figure}[h]
	\includegraphics[]{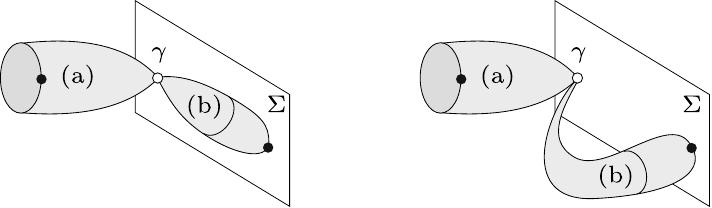}
	\caption{When $\gamma$ is a type \IV b orbit, the (b)-curve may or may not fall entirely within $\Sigma$.}
	\label{fig:ivb}
\end{figure}

Next, recall that $$(A+B)\cdot[\Sigma]=d$$ because $A+B$ is homologous to a Maslov index 2 disk $D$ one started with, see (\ref{eq:mu_and_intersec}). 
In both cases, it follows that 
$$A
\cdot[\Sigma]\le d,$$ or equivalently $\mu(A)\le 2$. Therefore 
$$\omega(A)\le 2\lambda$$ where $\lambda$ is the monotonicity constant.
Now one uses a version of the standard energy estimate, see e.g.~\cite[2.4]{Ri13}, re-cast in a slightly different fashion. Let $$u(s,t)\co [0,+\infty)\times S^1\to X$$ be an (a)-curve. Denote by $\gamma'\subset L$ the boundary loop of $u$, i.e.~$u(0,t)$. 
Recall that the (a)-curve solves the $s$-independent Floer's equation with the Hamiltonian $H_l=H_l^{(a)}$. Let $X$ be its Hamiltonian vector field. Then one has:
\begin{equation}
\label{eq:act_est}
\begin{array}{l}
0\le E(u)=\textstyle\int|\bd_s u|^2 ds\wedge dt\\
=\textstyle\int\omega(\bd_su,\bd_t u-X)ds\wedge dt\\
=\omega(A)-\textstyle\int\omega(\bd_su,X)ds\wedge dt\\
=\omega(A)-\textstyle\int dH_l(\bd_s u)ds\wedge dt\\
=\omega(A)-\textstyle\int_{\gamma}H_l(t)+\textstyle\int_{\gamma'}H_l(t).
\end{array}
\end{equation}
The last step uses the fact that $H_l$ is $s$-independent.
Recall that $\gamma\subset \Sigma$ is the constant type \IV b periodic orbit. By construction, $H_l|_\Sigma$ is a perturbation of the constant function equal to $l$, see (\ref{eq:h_on_bdle}) and Subsection~\ref{subsec:s_shape} where $H_l$ was defined. Therefore:
$$
\textstyle\int_{\gamma}H_l(t)\ge l-\epsilon
$$
for a small $\epsilon$.
Also by construction, $H_l$ is a perturbation of the zero function in $X_-$ (the region of $M$ away from the collar), and one can assume $L\subset X_-$. So
$$
\textstyle\int_{\gamma'}H_l(t)\le \epsilon.
$$
Putting the estimates together, one obtains:
$$
0\le 2\lambda-l+2\epsilon.
$$
For $l$ larger than $2\lambda+2\epsilon$, this is a contradiction.
\end{proof}	

\begin{lemma}
	\label{lem:pos_area}
Suppose the (b)-curve is contained inside $\Sigma$, then it defines a class $B\in H_2(X)$ such that $\omega(B)\ge 0$.	
\end{lemma}		
	\begin{proof}
 The argument uses an energy estimate similar to (\ref{eq:act_est}). Let $u$ be such  a (b)-curve defining a class $B\in H_2(X)$. The domain of $u$ is bi-holomorphic to $\C$; let us puncture the point $0$ in the domain and reparametrise the resulting domain to $\R\times S^1$ straightening the co-ordinate vector fields $\bd_s$, $\bd_t$. Recall that $u$ solves Floer's equation with an $s$-dependent Hamiltonian $H^{(b)}_l$ as in (\ref{eq:H_a_and_b}). For the purpose of this lemma, denote
$$H=H^{(b)}_l.$$
Then the type \IV b orbit $\gamma$ is the $s\to-\infty$ asymptotic of $u$, and the point $q=u(0)\in\Sigma$ is the $s\to+\infty$ asymptotic (seen as a removable singularity). Accounting for the $s$-dependence, the estimate analogous to (\ref{eq:act_est}) is:
$$
\begin{array}{l}
	0\le E(u)=\textstyle\int_{\R\times S^1}|\bd_s u|^2 ds\wedge dt\\
	=\ldots\\ =\omega(B)-\textstyle\int dH(\bd_s u)ds\wedge dt\\
	=\omega(B)-\textstyle\int_{\gamma}H(-\infty,t)+\textstyle\int_{q}H(+\infty,t)+\int(\bd_sH)ds\wedge dt
\end{array}
$$

Recall that $\bd_s H\le 0$, and by the arrangements made in Subsection~\ref{subsec:stretch} (see item $B$ especially) there is a sub-annulus in the domain of the (b)-curve where $\bd_sH|_\Sigma$ is constant with respect to the variables in $\Sigma$: it only depends on $s$, and integrates to $\epsilon-l$ over the sub-annulus. This, together with the monotonicity of $H$, implies that
$$
\textstyle\int(\bd_sH)ds\wedge dt\le \epsilon-l.
$$
 By the construction  of $H$ one also has: $$\textstyle\int_{\gamma}H(-\infty,t)=\int_{\gamma} H_l(t)\ge l-\epsilon,\quad \textstyle\int_{q}H(+\infty,t)=0.$$ The estimate becomes
 $$
 0\le \omega(B)-(l-\epsilon)+(\epsilon-l)=\omega(B)+2\epsilon.
 $$
 For sufficiently small $\epsilon$ it follows that $\omega(B)\ge 0$ since the areas are discrete by monotonicity.
\end{proof}

\begin{remark}
	\label{rem:b_const}
The only case in the  proof of Lemma~\ref{lem:not_4b} which is ruled out by crucially using the condition that $l$ is sufficiently large (larger than the area of a Maslov index~2 disk) is the case when the (b)-curve defined a class zero area. Note that in this case the (b)-curve is constant, with a constant type \IV b orbit asymptotic. (Indeed, the (b)-curve has to be a semi-infinite Morse flowline in $\Sigma$ and must be rigid, hence it is constant.)

Indeed, revisiting the proof above one sees that otherwise $B\cdot [\Sigma]\ge 1$, so
$$B\cdot [\Sigma]\ge d$$
because $[\Sigma]$ is divisible by $d$.
Applying positivity of intersections---Lemmas~\ref{lem:pos_inter} and~\ref{lem:pos_inter_asympt}---to the (a)-curve, one
obtains $A\cdot[\Sigma]\ge 1$. Indeed, the contribution to the intersection number from the constant asymptotic $\gamma$ is at least $\lfloor \alpha\rfloor+1= 1$.
Recalling that $(A+B)\cdot[\Sigma]=d$,
one gets a contradiction without using the fact that $l$ is large.
\end{remark}

\subsection{No extra bubbles}
We  rewind to prove Lemma~\ref{lem:a_b} which asserts that a broken configuration cannot have any other components (disk or sphere bubbles, Floer cylinders) except for the (a)- and the (b)-curve.

\begin{figure}[h]
	\includegraphics[]{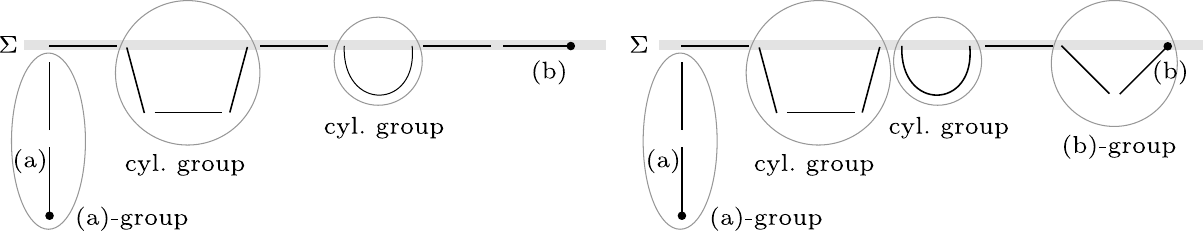}
	\caption{A broken configuration after the sphere and disk bubbles are dropped out. It is composed of the (a)- and the (b)-curve, plus Floer cylinders. The figure shows various positions with respect to $\Sigma$.}
	\label{fig:other_cyl}
\end{figure}

\begin{proof}[Proof of Lemma~\ref{lem:a_b}]
	The disk bubbles must have Maslov index $\ge 2$ by monotonicity, therefore have intersection number at least $d$ with $[\Sigma]$ by (\ref{eq:mu_and_intersec}).
	The sphere bubbles, whether or not contained in $\Sigma$, also have homological intersection at least $d$ with $[\Sigma]$. This is obvious for the bubbles not in $\Sigma$, and  true for the bubbles in $\Sigma$ for the following reason: they have positive area, therefore have positive Chern class in $X$, but $\Sigma$ is dual to $dc_1(X)$. So we record that if there is at least one disk or sphere bubble, the rest of the curve has homological intersection number $\le 0$ with $\Sigma$. If there are no disk or sphere bubbles, the homological intersection is still $d$.
	
	Let us forget the disk and sphere bubbles (if they exist) and look at the rest of the broken curve. It is composed of the (a)-curve, the (b)-curve and a string of Floer cylinders between them, attached to each other along asymptotic 1-periodic orbits of $H_l$.
	
	{\it Case 1.
		There is at least one asymptotic orbit which has type \IV b.}  In this case some Floer cylinders and/or the (b)-curve may be contained in $\Sigma$, see Figure~\ref{fig:other_cyl}  where the curves are represented by segments.
	The idea is to group the pieces of the broken curve into collections starting and ending at type \IV b asymptotics, shown by circles in Figure~\ref{fig:other_cyl}. We will show that every group has non-negative intersection with $\Sigma$.
	The argument is independent on whether there were any additional disk or sphere bubbles, and is an expansion on  the proof of Lemma~\ref{lem:not_4b}.
	
	As a general observation, those Floer cylinders which are contained in  $\Sigma$ have two type  \IV b asymptotics and  compactify to closed cycles $A\in H_2(X)$. Such cycles have non-negative area,  by a similar argument as was used to prove (see Lemmas~\ref{lem:not_4b},~\ref{lem:pos_area}) that the (b)-curves contained in $\Sigma$ have non-negative area. Therefore their classes satisfy
	\begin{equation}
	\label{eq:int_floer_in_s}
	A\cdot [\Sigma]\ge 0 \qquad (A=\text{Floer cylinder in $\Sigma$)}
	\end{equation}
	
	Next, we call a connected union of Floer cylinders, none of which is contained in $\Sigma$, beginning and ending at a type \IV b asymptotic a \emph{cylinder group}. See the circled groups in Figure~\ref{fig:other_cyl}. Compactifying by the type \IV b asymptotics, one sees that a cylinder group defines a class in $A\in H_2(X)$. We claim that the class of a cylinder group satisfies
	\begin{equation}
	\label{eq:cyl_group}
	A\cdot[\Sigma]\ge d \qquad (A=\text{cylinder group)}
	\end{equation}
	It suffices to show that $A\cdot[\Sigma]\ge 1$, because $[\Sigma]$ is divisible by $d$.
	And indeed, the contribution from the beginning and the ending  asymptotic of a group to $A\cdot[\Sigma]$ is in total at least
	$$
	\lfloor\alpha\rfloor+1-\lfloor\alpha\rfloor=1
	$$
	by Lemma~\ref{lem:pos_inter_asympt} (since one asymptotic is positive and the other one is negative), regardless of $\alpha$. The contribution from any other geometric intersection is positive by Lemma~\ref{lem:pos_inter}.
	
	{\it Case 1i. The (b)-curve is contained in $\Sigma$.} In this case we have shown in Lemma~\ref{lem:pos_area} that the class $A\in H_2(X)$ of the (b)-curve satisfies
	\begin{equation}
	\label{eq:basis_in_sigma}
	A\cdot [\Sigma]\ge 0 \qquad (A=\text{(b)-curve in $\Sigma$)}
	\end{equation}

	{\it Case 1ii. The (b)-curve is not contained in $\Sigma$.}
Consider the last group of curves not contained in $\Sigma$, beginning with a negative type \IV b orbit and ending with the (b)-curve; see Figure~\ref{fig:other_cyl}.
	This gives a cycle $A\in H_2(X)$ such that
	\begin{equation}
	\label{eq:group_b}
	A\cdot[\Sigma]\ge d \qquad (A=\text{group with the (b)-curve)}.
	\end{equation}
	Indeed the intersection at the negative type \IV b orbit at least $-\lfloor\alpha\rfloor= 0$, by Lemma~\ref{lem:pos_inter_asympt} and our arrangement about $\alpha$. But we also have a strictly positive intersection due to the incidence condition for the (b)-curve, so the desired estimate follows.
	
	{\it Case 1, common conclusion.} 
	Consider 
	the remaining group of curves beginning with the (a)-curve and ending at the first positive type \IV b asymptotic; see Figure~\ref{fig:other_cyl}. We call it an (a)-group and it defines a relative homology class $A\in H_2(X,L)$. Recall that the total intersection of all punctured curves with $\Sigma$ was at most $d$, and the intersections of the curves not in the (a)-group  are estimated 
	by (\ref{eq:int_floer_in_s}), (\ref{eq:cyl_group}), (\ref{eq:basis_in_sigma}) and (\ref{eq:group_b}). So for the class $A$ of an (a)-group one has
	$$
	A\cdot[\Sigma]\le d \qquad (A=\text{(a)-group)}
	$$
	consequently, $\mu(A)\le 2$ and 
	$$
	\omega(A)\le 2\lambda.
	$$
	Then one uses an analogue of (\ref{eq:act_est}) to get a contradiction whenever $l\ge 2\lambda$. (This holds regardless of whether one had to forget any disk or sphere bubbles in the beginning.)
	Note that one has to apply estimate (\ref{eq:act_est}) to several curves comprising the group and add them up.

	{\it Case 2. There are no type \IV b asymptotics.} In this case each broken curve in the broken configuration has non-negative intersection number with $\Sigma$ on its own. Now, we claim there can be no disk or sphere bubbles: indeed, after forgetting them,  the rest of the curve would have intersection $\le 0$ with $\Sigma$; on the other hand, the latter intersection number is at least $1$ due to the incidence condition of the (b)-curve with $\Sigma$. Floer cylinders are ruled by a simple argument involving the dimensions of moduli spaces, using the fact the curves are not contained in $\Sigma$ and are therefore generically regular. 
\end{proof} 	

\subsection{Ruling out type \IV a orbits}
Recall that type \IV a orbits are constant orbits near $\Sigma$ and not contained in $\Sigma$.

\begin{lemma}
	\label{lem:disjoint_from_sigma}
The (a)-curve of the broken configuration is disjoint from $\Sigma$. 	
\end{lemma}

\begin{proof}
Suppose first that the asymptotic orbit $\gamma$ has type \IV a or \I: then it is constant. Consider the homology classes $A\in H_2(X,L)$ and $B\in H_2(X)$ of the compactified (a)- and (b)-curves.
In this case the (b)-curve is automatically not contained in $\Sigma$ (because $\gamma$ is not), so the whole configuration looks like in Figure~\ref{fig:iva}.
 Arguing as before,  one obtains $B\cdot[\Sigma]\ge d$ and therefore $A\cdot[\Sigma]=0$. By positivity of intersections (Lemma~\ref{lem:pos_inter}), this means that the (a)-curve is disjoint from $\Sigma$.

\begin{figure}[h]
	\includegraphics[]{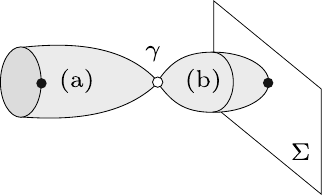}
	\caption{The (a)- and (b)-curve when $\gamma$ has type \IV a.}
	\label{fig:iva}
\end{figure}

Suppose that $\gamma$ has type \II\ or \III; then it belongs to the collar $[1-\delta,1+\delta]\times \bd M$. Denote 
$$
\label{eq:N}
N'=X\setminus \left(M\cup([1-\delta,1+\delta]\times \bd M)\right),
$$
which is diffeomorphic to $N=X\setminus M$. Then the images of the (a)-curve and (b)-curve define relative homology classes
$$
A,B\in H_2(X, N').
$$
We know that $B\cdot[\Sigma]\ge 1$ by positivity of intersections (Lemma~\ref{lem:pos_inter}) and the fact that it intersects $\Sigma$ at least once by the incidence condition $u(0)\in\Sigma$ for the (b)-curve. Then by Lemma~\ref{lem:div_by_k}, $ B\cdot[\Sigma]\ge d$. As  above,  $(A+B)\cdot[\Sigma]=d$ so $A\cdot[\Sigma]=0$ which implies the result. 
\end{proof}	

\begin{lemma}
	\label{lem:not_4a}
For sufficiently large $l$,
a limiting orbit $\gamma$ cannot be of type \IV a.
\end{lemma}

\begin{proof}
Let $\theta$ be the Liouville form on $X\setminus\Sigma$ making $L$ exact.
By definition, the action of a loop $\gamma\co S^1\to X\setminus\Sigma$ is:
$$
A(\gamma)=-\int_{S^1}\gamma^*\theta+\int_{S^1}H\circ\gamma.
$$ 
When $\gamma$ is a periodic orbit of $H_l$ of type \IV a one has:
$$
A(\gamma)\ge l-\epsilon.
$$
This is because $\gamma^*\theta=0$ for a constant orbit, and $H_l(\gamma)$ is by construction close to $l$ in the region containing type \IV a orbits.

Looking at the (a)-curve, let $\gamma'\subset L$ be its boundary loop. Then $$A(\gamma')\le \epsilon,$$ given that $\theta|_L=0$ (by exactness) and $H_l|_L$ is small (by construction); compare with the proof of Lemma~\ref{lem:not_4b}. However, since the Floer equation for the (a)-curve is $s$-independent, the standard action estimate says that
$$A(\gamma')\ge A(\gamma).$$ This gives a contradiction.
\end{proof}

\begin{remark}
At this point, we know that $\gamma$ must be of type \I, \II\ or \III, for $l$ greater than the area of a Maslov~index~2 disk. In fact, the only place where this condition is needed is discussed in Remark~\ref{rem:b_const}; all other arguments work for a small $l$ as well.

It is instructive to understand what happens when one takes the height $l$ for $H_l$ (see Figure~\ref{fig:s_shape}) to be very small, and keeps $\alpha$ a small positive number: $\lfloor\alpha\rfloor=0$. This way one can arrange $H_l$ to have so small slopes uniformly that it acquires no 1-periodic orbits of type \II, \III\ at all. One can also show that the (b)-curves  do not output type \I\ orbits, at least at the cohomology level, see Lemma~\ref{lem:non_const} below.
Revisiting the above proofs in this case, one concludes that a limiting orbit $\gamma$ of the broken curve must be of type \IV b, and moreover the (b)-curve must be constant. 

The upshot is that when $l$ is small, the breaking of Maslov~2 disks into the (a)-curve and the (b)-curve is essentially trivial, in the sense that the (b)-curve is constant and  the (a)-curve `looks like the initial disk'. We learn that domain-stretching does not achieve anything structurally useful in this case.

\end{remark}

\subsection{Ruling out type \III\  orbits} The next step is to show that type \III\ orbits also cannot arise as an asymptotic orbit of a broken configuration of an (a)- and a (b)-curve.

\begin{lemma}
	\label{lem:not_3}
A limiting orbit $\gamma$ cannot be of type \III. 
\end{lemma}

\begin{proof}
It is enough to look at the (a)-curve. Lemma~\ref{lem:prohibited_trj} says that the (a)-curve near $\gamma$ `enters the region to the right of $\gamma$' in the collar co-ordinate, i.e.~towards $X_+$ which is concave. It contradicts the no-escape lemma whose particular case is summarised below. 
\end{proof}

\begin{remark}
We learned this way of combining the no-escape lemma with Lemma~\ref{lem:prohibited_trj} from \cite[Section 2.3]{CO18}.
The same argument is essentially used in \cite[Proposition~4.4]{Gu17} and \cite[Proposition~2.11]{La16}.
\end{remark}

The no-escape lemma, whose particular case appears below, is due to Abouzaid and Seidel 
\cite[Lemmas~7.2 and~7.4]{ASei10}; see also \cite[Lemma~19.5]{Ri13} and \cite[Lemma~2.2]{CO18}.
It generalises the maximum principle for Floer solutions: while the maximum principle only applies inside a Liouville collar (or the symplectisation) of a contact manifold, 
the no-escape lemma allows an arbitrary Liouville cobordism instead of a collar. Recall that $J$ is of contact type near the collar.

\begin{lemma}[No-escape lemma]
	\label{lem:no_escape}
Let $u\co[0,+\infty)\times S^1\to X\setminus \Sigma$ solve the Floer equation for the (a)-curve, with boundary on $L$ and asymptotic orbit $\gamma$. Observe the assumption that $u$ avoids $\Sigma$. 

If $\gamma$ is of type \I\ or \II\ then the image of $u$ is contained in $M$; if $\gamma$ is of type \III\ then its image is contained in $M\cup[1,1+r_0]\times \bd M$ where $\gamma\subset \{1+r_0\}\times \bd M$, $r_0<\delta$, see Figure~\ref{fig:s_shape}.
\end{lemma}

\begin{proof}
Recall that $H_l$ is the \S-shaped Hamiltonian used in the Floer equation for the (a)-curve. 
First, we explain why the image of $u$ must lie inside $M\cup([1,1+\delta]\times \bd M)$. Observe that both $L$ and $\gamma$ lie in this domain.
Moreover, $H_l$ is close to a constant function near $\{1+\delta\}\times \bd M$.
Let us introduce the following additional requirement on the profile of  $H_l$, recall Figure~\ref{fig:s_shape}. 
Recall that $H_l$ is a function of the collar co-ordinate on $[1-\delta,1+\delta]\times \bd M$. We require that
there exists a number $\delta_0$ close to $\delta$ such that $H_l$ is a \emph{linear} function of \emph{small slope} in a neighbourhood of $\{1+\delta_0\}\times \bd M$, and that the slopes of $H_l$ are small uniformly on $[1+\delta_0,1+\delta]\times\bd M$ and  $H_l$ has no periodic orbits in that region. See Figure~\ref{fig:h_del-lin} for a rough sketch.
It follows that $\gamma$ necessarily lies `to the left' of $\{1+\delta_0\}\times\bd M$.

Assume that the (a)-curve enters the region $[1+\delta_0,1+\delta]\times\bd M$.
To get a contradiction, one can apply the no-escape lemma as stated in \cite[Lemma~D.3 and Remark~D.4(2)]{Ri13}, using the contact hypersurface $\{1+\delta_0\}\times\bd M$. The lemma requires that the Hamiltonian be linear in the collar co-ordinate, which has been arranged. It follows that the (a)-curve is contained in $M\cup[1,1+\delta_0]\times \bd M$.

Next, the inclusion into the smaller subdomain $M$ or $M\cup [1, 1+r_0]$ follows from the maximum principle which, roughly, says that $u$ cannot extend inside the collar $[1-\delta,1+\delta]\times \bd M$ `to the right' of all of its asymptotic or boundary conditions. The maximum principle applies whenever  $H$ has the form $h(r)$ on the collar, $h'(r)\ge 0$ \cite[Lemma~19.1]{Ri13}, which certainly holds in our case.

Formally speaking, the above linearity condition for $H_l$ near $\{1+\delta_0\}\times\bd M$ should have been mentioned in the initial setup of $H_l$ in Subsection~\ref{subsec:s_shape}. However, because this is a minor technicality and there are alternative arguments (see the remark below) which do not require the linearity, we decided to keep the extra condition within this proof.
\end{proof}	

\begin{figure}[h]
	\includegraphics[]{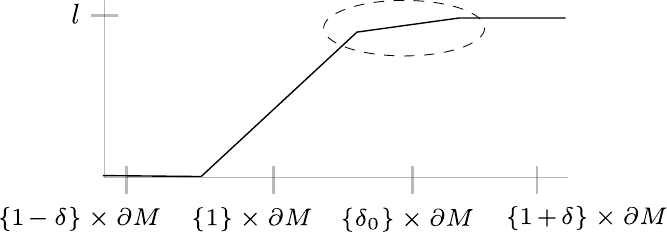}
	\caption{The \S-shaped Hamiltonian $H_l$ near the collar, with an additional linear part of small slope.}
	\label{fig:h_del-lin}
\end{figure}

\begin{remark}
The two-step proof above seems like a necessary technicality, as one cannot apply the no-escape lemma directly to the hypersurface $\{1+r_0\}\times \bd M$. The reason is that our Hamiltonian $H_l=h(r)$ depends on the radial co-ordinate non-linearly inside the collar $[1-\delta,1+\delta]\times \bd M$. The no-escape lemma in this setting still exists \cite[Lemma~19.5]{Ri13}, but it requires the condition that 
\begin{equation}
\label{eq:ineq_ri}
-r_0h'(r_0)+h(r_0)\le 0
\end{equation}
 which is not necessarily satisfied in our case. In the above proof, one of the possible workarounds was used. An alternative solution is to notice that $-r_0h'(r_0)+h(r_0)=A(\gamma)$, and if $A(\gamma)>0$ the (a)-curve cannot exist, as one must have $\epsilon\ge A(\gamma')\ge A(\gamma)$, see the proof of Lemma~\ref{lem:not_4a}. 
\end{remark}

\begin{corollary}
\label{cor:a_curve}
A limiting orbit $\gamma$ either has type \I\ or type \II.
The image of the (a)-curve of any broken configuration lies in $M$.
\end{corollary}

\begin{proof}
	This follows from Lemma~\ref{lem:no_escape}, combined with the previous results which guarantee that the (a)-curve has asymptotic of type \I\ or \II.
\end{proof}

\subsection{The degree is zero}
\label{subsec:grade}
Recall that $\Sigma\subset X$ is a Donaldson divisor of degree~$d$. Let $\zeta$ be a trivialisation of $K_M$ such that $\zeta^d$ is the natural trivialisation of $(K_{X\setminus \Sigma})^d$; it exists by the hypothesis of graded-compatibility.
Let $(B,\bd B)\subset (X,M)$ be a 2-chain with boundary. By construction,  one has the following identity between its relative Chern class with respect to $\zeta$, and the intersection number with $\Sigma$:
\begin{equation}
\label{eq:c_1}
c_1(B,K_X|_B,\zeta|_{\bd B})=\tfrac 1 d [B]\cdot[\Sigma].
\end{equation}
\begin{lemma}
	\label{lem:degree}
A limiting orbit $\gamma$ has degree zero in the above trivialisation of $K_M$.
\end{lemma}

\begin{proof}
	Let $|\gamma|\in \Z$ be the degree. Our grading conventions (explained in the introduction) are such that the dimension of the moduli space of (a)-curves \emph{without} the boundary point constraint $u(1)=\pt\in L$ equals
	$$
	n-|\gamma|,
	$$
	using the fact that the (a)-curves are contained in $M$. The moduli space of the (b)-curves without the interior point constaint $u(1)\in\Sigma$ is
	$$
	|\gamma|+2,
	$$	
	where $2$ is twice the relative Chern number from (\ref{eq:c_1}), using the fact that the (b)-curve has intersection number $d$ with $\Sigma$. Adding the incidence conditions and using regularity, one arrives at the conditions:
	$$
	n-|\gamma|\ge n, \quad |\gamma|+2\ge 2.
	$$
	It follows that $|\gamma|=0$.
\end{proof}

\subsection{The Borman-Sheridan class}
Let $CF^*(H_l)$ be the Floer complex generated by all periodic orbits of $H_l$; it carries the usual Floer differential. Denote by
$$CF^*_{\IandII}(H_l)\subset CF^*(H_l)$$ 
the subspace generated by periodic orbits of $H_l$ in $X$ having type \I\  or \II. 
We define 
$$\BS_l\in CF^0_\IandII(H_l)$$ 
to be the chain  counting the output asymptotics $\gamma$ of all rigid maps $u\co \C\to X$ which:
\begin{itemize}
	\item solve Floer's equation with the Hamiltonian perturbation $H^{(b)}_l$ from (\ref{eq:H_a_and_b}),
	\item 
 satisfy the incidence condition $u(0)\in\Sigma$,
 \item have homological intersection number $d$ with $\Sigma$,
 \item and have asymptotic orbit $\gamma$ of type \I\ or \II. 
\end{itemize}
We call such maps the \emph{(b)-curves}.
The last condition has to be added explicitly, and is not guaranteed otherwise. (Above, we only proved that $\gamma$ must of type \I\ or \II\ for configurations arising from the breaking of a Maslov index~2 disk; we were using the existence of an (a)-curve as well.)

  It follows as in the proof of Lemma~\ref{lem:degree} that the outputs of the (b)-curves have degree zero, so indeed $\BS_l\in CF^0_\IandII(H_l)$.

\begin{proposition} 
	\label{prop:bs_def}
	In our setting, the following holds.
\begin{itemize}
	\item[(i)]	$CH^*_\IandII(H_l)$ is a complex with respect to the differential $d^0_l$ counting only those Floer cylinders which run between the type \I\ and\ \II\ orbits, and moreover do not intersect $\Sigma$.
	\item[(ii)] For $l'>l$, there are chain maps $$c^0_{l,l'}\co CF^*_\IandII(H_l)\to CF^*_\IandII(H_{l'})$$
	counting only those continuation map solutions $CF^*(H_l)\to CF^*(H_{l'})$
	which run between the type \I\ and \II\ orbits, and
	moreover do not intersect $\Sigma$. (We use interpolating Hamiltonians that are monotonically non-increasing in $s$.)
	\item[(iii)] The direct limit of the cohomologies of $CH^*_\IandII(H_l)$ with respect to $c^0_{l,l'}$ is isomorphic to the symplectic cohomology of $M$: 
	$$SH^*(M)\cong \lim\limits_{l\to+\infty}HF^*_\IandII(H_l).$$
	\item[(iv)] The elements $\BS_l\in CF^0_\IandII(H_l)$ are $d^0_l$-closed, and their homology classes are respected by the maps $c^0_{l,l'}$. Therefore, they define an element 
	$$\BS\in SH^0(M)$$
	which we call the Borman-Sheridan class. It depends both on $M$ and its Liouville embedding into $X\setminus\Sigma$, but is invariant of Liouville deformations of these data.
\end{itemize}	
\end{proposition}

\begin{proof}
The proofs use same the  ingredients that have already been used above, so we shall be brief.

For (i), one has to show that $(d_l^0)^2=0$. It is enough  to prove that Fredholm index~1 Floer trajectories connecting type \I,\ \II\ orbits, and which have algebraic intersection zero with $\Sigma$, cannot break at type \III\ or type \IV\ orbits.  The proof is analogous to Lemmas~\ref{lem:not_4b},~\ref{lem:not_4a},~\ref{lem:not_3}, also employing an analogue of Lemmas~\ref{lem:no_escape} and~\ref{lem:a_b}. For example, let us check how type \IV b orbits are ruled out. If a Floer cylinder as above breaks at a type \IV b orbit, an argument similar to Lemma~\ref{lem:not_4b} shows that both broken cylinders have non-negative algebraic intersection with $\Sigma$, and at least one cylinder has positive algebraic intersection, which is a contradiction.

For (ii), one needs to rule out the same types of breaking for the continuation maps $c^0_{l,l'}$. This follows the same steps as above, with  some changes standard  in Floer theory: for example, the energy estimate used in Lemma~\ref{lem:not_4a} still holds if we use the continuation map Hamiltonian
$H(s,t,x)$ satisfying $\bd_s H\le 0$.

Using (i) and (ii),
it follows that the direct limit in (iii) makes sense. 
We claim that there is a chain level isomorphism:
\begin{equation}
\label{eq:sh_chain_iso}
HF^*_\IandII(H_l)\cong HF^*(\hat H_l)
\end{equation}
where $\hat H_l$ were introduced in (\ref{eq:h_sh}) as the Hamiltonians computing $SH^*(M)$. First, recall that
\begin{equation}
\label{eq:h_eq_hhat}
H_l
|_{S^1\times M}\equiv \hat H_l,
\end{equation}
and the periodic orbits of $H_l$ contained in $M$ are precisely the type \I\ and \II\ orbits. 
It follows that the generators of (\ref{eq:sh_chain_iso})  coincide. Moreover, solutions contributing to $d^0_l$ in fact belong to $M$ by the no-escape lemma (analogous to Lemma~\ref{lem:no_escape}), so the differentials defining both sides of (\ref{eq:sh_chain_iso})  also coincide. Compare with a similar argument in e.g.~\cite[Proposition~4.5]{GP16}.

Again by a version of the no-escape lemma, (\ref{eq:sh_chain_iso})
intertwines $c^0_{l,l'}$ with the standard continuation maps $CF^*(\hat H_l)\to CF^*(\hat H_{l'})$. And since $$SH^*(M)=\lim_{l\to+\infty} HF^*(\hat H_l)$$ by definition, one immediately obtains (iii).

To prove (iv),   consider the moduli space like that defining $\BS_l$, but 1-dimensional rather than 0-dimensional. One checks that its boundary points correspond to $d_l^0(\BS_l)$. To do so, following the above steps one shows that the 1-parameter families of curves break along orbits which are of type \I\ or \II, and the `left' parts of the broken curves are confined to $M$  and compute the $d^0_l$-differential. (No extra bubbles in $\Sigma$ are possible by a version of Lemma~\ref{lem:higher_bubble_sigma} and Lemma~\ref{lem:a_b}.)
\end{proof}

\subsection{Concluding the proof}
Let us return to the preceding steps where we applied a domain-stretching procedure to Maslov index~2 holomorphic disks and ended up in a broken configuration consisting of the (a)- and (b)-curve. Let us  look at the (a)-curves.

By Corollary~\ref{cor:a_curve}, the (a)-curves are contained in $M$, and the equation they solve is precisely the equation for $\CO_l$ (see Subsection~\ref{subsec:sh}) because:
$$
H_l|_{S^1\times M}\equiv \hat H_l,
$$
where the two sides were defined in (\ref{eq:H_l_S}) and (\ref{eq:h_sh}); see also (\ref{eq:H_a_and_b}).
So by Proposition~\ref{prop:gluing_broken}, 
under   isomorphism (\ref{eq:sh_chain_iso}) we obtain: 
$$\CO_l(\BS_l)= d\cdot  W_L(\rho)\cdot 1_L\in HF^*_M(\LL,\LL).$$
Passing to the limit $l\to+\infty$, it is seen that
$$
\CO(\BS)= d\cdot  W_L(\rho)\cdot 1_L\in HF^*_M(\LL,\LL).
$$
This completes the proof of Theorem~\ref{th:factor}.

\subsection{Higher deformation classes}
\begin{proof}[Proof of Theorem~\ref{th:factor_higher}] The proof is entirely analogous to the proof of Theorem~\ref{th:factor}. 
The first step of the proof---introducing the interior marked point---is unnecessary since it is already built in the definition of higher disk potentials. Apply the same domain-stretching procedure; the (b)-curve inherits the tangency condition eating up the total homological intersection number $k$ with $\Sigma$. Repeating the above arguments, one finds that the limiting orbit $\gamma$ is of type \I\ or \II, and that the (a)-curve is confined to $M$ thus computing the closed-open map. We define the class $\D_k$ to be the count of the outputs of the (b)-curves, namely
we introduce 
$$(\D_k)_l\in CF^0_\IandII(H_l)$$ 
to be the chain  counting the output asymptotics $\gamma$ of all rigid curves $u\co\C\to X$ which 

\begin{itemize}
	\item solve Floer's equation with the Hamiltonian perturbation $H^{(b)}_l$ from (\ref{eq:H_a_and_b}),
	\item 
	satisfy  order $k$ tangency condition at the point $u(0)\in\Sigma$ (this makes sense because the Hamiltonian perturbation is zero in a neighbourhood of $0$ in the domain, see e.g.~Figure~\ref{fig:abc}),
	\item have homological intersection number $k$ with $\Sigma$,
	\item and have asymptotic orbit $\gamma$ of type \I\ or \II. 
\end{itemize}
After this, the class $\D_k\in SH^0(M)$ is defined analogously to Proposition~\ref{prop:bs_def},
and the conclusion of the proof is similar.
\end{proof}

\section{Anticanonical divisors and partial compactifications}
\label{sec:antican}
This section covers two remaining   topics: the proof of Proposition~\ref{prop:bs_antican} which computes of the Borman-Sheridan class (and higher deformation classes) in the complement of an anticanonical divisor; and a version of Theorem~\ref{th:factor} in a setting where $X$ is allowed to be non-compact.

\subsection{Equivariance}
We begin with an observation about the Borman-Sheridan class. 
Recall that the outputs of the (b)-curves appearing in the previous section which define the elements $$\BS_l\in CF^0_\IandII(H_l),$$ are Hamiltonian orbits of $H_l$ of type \I\ and \II.
The following lemma will be helpful soon.

\begin{lemma}
	\label{lem:non_const}
	The element $\BS_l$ is a linear combination of  type \II\ orbits, i.e.~type \I\ orbits do not contribute to $\BS$.
\end{lemma}

For a proof we shall assume that $H_l$ and $H_{n,l}$ used in Section~\ref{sec:proofs} are autonomous Hamiltonians.  Before proceeding, let us review how the moduli problem for the (a)- and the (b)-curves gets modified in the autonomous setting.

According to the standard modification of Floer theory in the $S^1$-Morse-Bott setting \cite{BO09,BEE}, one takes the following generators to define the complex $CF^0_\IandII(H_l)$ when $H_l$ is autonomous. Each type \I\ (constant) orbit gives one generator; for each $S^1$-family of type \II\ orbits we pick one parameterised orbit $\gamma$ and introduce two formal generators $\hat \gamma$, $\check \gamma$. (Our notation follows \cite{BEE}; in \cite{BO09} these would be $\gamma_p$ and $\gamma_q$.) It is natural to treat a constant orbit $\gamma$  as a $\hat\gamma$-orbit. The grading rule is $|\check\gamma|=|\hat \gamma|+1$, and for a constant orbit $\gamma$, the degree of $|\hat \gamma|$ is the Morse index. After a non-autonomous perturbation of the Hamiltonian, we will see two type \II\ periodic orbits corresponding to each $\gamma$, with degrees matching the ones of $\hat \gamma$ and $\check \gamma$.

The moduli problems involving Floer solutions asymptotic to the periodic orbits of $H_l$ acquire the following modification. Each cylindrical end of a curve needs to be equipped with an asymptotic marker (a point in $S^1$, where $S^1$ is seen as the boundary-at-infinity of the cylindrical end). Having asymptotic $\hat\gamma$ at a negative puncture or asymptotic $\check \gamma$ at a positive puncture means that the asymptotic marker is required to go to the initial point $\gamma(0)$ of the orbit $\gamma$, where $\gamma$ has the parameterisation fixed in advance. In the two other scenarios, the asymptotic marker is unconstrained. These rules in particular  concern the (a)- and the (b)-curves appearing in Section~\ref{sec:proofs}.

\begin{proof}[Proof of Lemma~\ref{lem:non_const}]
	Consider the $S^1$-action on the domain of the (b)-curves rotating in $t$; it fixes the origin which is required to pass through $\Sigma$.
To understand whether this gives an $S^1$-action on the moduli space of (b)-curves defining the classes $\BS_l$, one must pay attention to the asymptotic markers; recall that the (b)-curves have a negative puncture. The answer is that there is an $S^1$-action on the part of the moduli space whose outputs are type \I\ orbits, or type \II\ orbits of form $\check \gamma$.
Indeed, for $\check \gamma$-asymptotics the negative asymptotic marker is unconstrained, so the moduli problem is still satisfied after rotation. For type \I\ orbits the asymptotic marker is required to go to $\gamma(0)$ but since $\gamma$ is constant it holds that $\gamma(0)=\gamma(t)$ for all $t$, and the moduli problem is again satisfied after rotation.

Consider those (b)-curves whose asymptotic is of type \I.
The $S^1$-action on them is non-trivial unless a curve in question is independent of $t$, in which case it is a flowline. But the (b)-curves cannot be flowlines because they are required to have intersection number~1 with $\Sigma$, and this number is zero for a flowline. So the $S^1$-action is non-trivial and the (b)-curves are not rigid.
On the other hand, the (b)-curves contributing to $\BS_l$ must be rigid, which means that the (b)-curves with a  type \I\ asymptotic do not exist.
\end{proof}

\begin{remark}
	\label{rmk:b_no_i}
The non-existence of (b)-curves with type \I\ asymptotic can be compared to the following argument. Consider such a curve; its asymptotic has degree~0 so it corresponds to  the minimum $p$ of a Morse function. In the adiabatic limit when the Hamiltonian perturbation goes to zero, such a curve converges to a Chern number~1 sphere passing through $p$ at a fixed point in the domain, and through $\Sigma$ at another fixed point in the domain. But rather than being rigid, this problem is has virtual dimension $-2$ by an easy dimension count: $2n+2c_1-6+4-2n-2=-2$. The adiabatic limit reveals the same extra $S^1$-symmetry as that used in the proof above.
\end{remark}

Recall that one can define $S^1$-equivariant symplectic cohomology $SH^*_{eq}(M)$ of a Liouville domain $M$, see \cite{Sei08,BO09,BO16}, fitting into an exact sequence
$$
\ldots \to SH^{*-2}_{eq}(M) \to SH^*_{eq}(M)\to SH^*(M) \to SH^{*-1}_{eq}(M)\to \ldots
$$
The above proof essentially  shows  the following proposition (which will not be used).

\begin{proposition}
The Borman-Sheridan class admits a lift to $SH^*_{eq}(M)$.\qed
\end{proposition}

\subsection{The Borman-Sheridan class in the complement of an anticanonical divisor}
We will now prove Proposition~\ref{prop:bs_antican}.
Let $\Sigma$ be a smooth anticanonical divisor (i.e.~$d=1$), and $M=X\setminus\Sigma$. We will show that $\BS=r$, the generator corresponding to a simple Reeb orbit around $\Sigma$, and that $\D_k=r^k$ is analogous.
By Lemma~\ref{lem:non_const} one knows that all asymptotics of the (b)-curves are of type \II.

Pick an $\epsilon$-neighbourhood $U(\Sigma)$ of $\Sigma$ and choose the \S-shaped Hamiltonian $H_l$ whose growth happens within this neighbourhood, see Figure~\ref{fig:H_l_div}.

\begin{figure}[h]
	\includegraphics[]{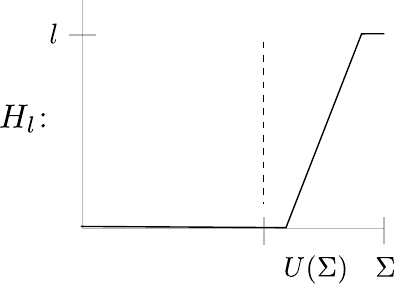}
	\caption{An \S-shaped Hamiltonian with growth in a neighbourhood of $\Sigma$.}
	\label{fig:H_l_div}
\end{figure}

The lemma below is a variation on the argument of Biran and Khanevsky \cite[Proposition~5.0.2]{Bk13}, with several major differences. As opposed to \cite{Bk13}, one has the freedom to make $U(\Sigma)$ as small as desired. Another difference is that while \cite{Bk13} work with purely holomorphic disks, our (b)-curves have a Hamiltonian perturbation, and we quote Bourgeous and Oancea \cite{BO09} for the relevant neck-stretching procedure.  

\begin{lemma}
In the chain model (\ref{eq:CH_chain_level_0}) for $SH^0(X\setminus\Sigma)$, the (b)-curves are all asymptotic to the orbit $r$; therefore the class $\BS$ is a multiple of $r$. 
Moreover, there is a choice of $J$ and $H_l$ such that
the (b)-curves are entirely contained in a neighbourhood of $\Sigma$.

The same is true for $\D_k$, where the orbit is $r^k$.
\end{lemma}

\begin{proof}	
We make the following arrangement:	 for any point $z$ in the domain of the (b)-curve, its Hamiltonian $H_l^{(b)}(z)|_{U(\Sigma)}$ is the sum of a function of the fibre distance to $\Sigma$ (denoted by $r$ in (\ref{eq:h_on_bdle})), and a $\pi$-pullback of some function on $\Sigma$. Compare with (\ref{eq:h_on_bdle}), which in particular has that form.
This choice can be first made for $H_l$ and extended to $H_l^{(b)}$. The effect of the arrangement is that if the (b)-curve lies in $U(\Sigma)$ (which we do not know yet), its projection to $\Sigma$ satisfies  a corresponding Floer equation.
Let $r^l$ be the output orbit of the (b)-curve.

We will now prove that the (b)-curve is contained in $U(\Sigma)$, and $k=l$.
Recall that our Hamiltonian is $C^1$-small near $\bd U(\Sigma)$. One applies neck-stretching along $\bd U(\Sigma)$ combined with  Hamiltonian slowdown as  explained in \cite[Section~5 and Figure~2]{BO09b}. Briefly speaking, neck-stretching modifies the almost complex structure in a small collar neighbourhood of $\bd U(\Sigma)$ by making the collar `long', and Hamiltonian slowdown modifies $H_l$ on that collar so that it has constant slope tending to zero. (The modification of $H_l$ stays within the class of \S-shaped Hamiltonians introduced in Subsection~\ref{subsec:s_shape}. It can be done compatibly with the arrangement in the beginning of the proof, if we stretch with respect to the `round' neighbourhood $U(\Sigma)$ whose boundary has the $S^1$-periodic  Reeb flow.)

\begin{figure}[h]
	\includegraphics{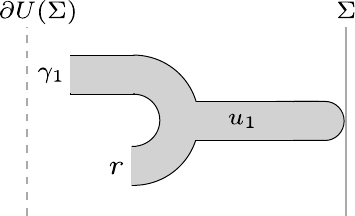}
	\caption{The part $u_1$ of the SFT limit of a (b)-curve which contains the initial Hamiltonian asymptotic orbit $\gamma$, and additional SFT punctures $\gamma_i$.}
	\label{fig:Neck_Stretching}
\end{figure}

By the SFT compactness theorem \cite{CompSFT03}, \cite[Proof of Proposition~5, Step~1]{BO09b},
if the (b)-curve is not eventually contained in $U(\Sigma)$ for the neck-stretching sequence of almost complex structures and $H_l$s, one gets a non-trivial holomorphic building in the SFT limit. Let $u_1\subset U(\Sigma)$ be the part of the building containing the original asymptotic Hamiltonian orbit $r^l$---this part is now constrained to $U(\Sigma)$. It has additional SFT punctures asymptotic to Reeb orbits $\gamma_1,\ldots,\gamma_q$, with multiplicity $k_i$.
Observe that $u_1$ is the only part of the building to lie in $U(\Sigma)$, because $u_1$ realises the homological  intersection number of the whole building with $\Sigma$.

By the similar observation as above, the projection $\pi(u_1)$ solves a Floer equation in $\Sigma$. Because $r$ and $\gamma_i$ project to points, $\pi(u_1)$ defines a class of a sphere; its Chern number in $\Sigma$ vanishes, because $\Sigma$ is anticanonical. Moreover, the $\gamma_i$ project to removable singularities, and $r^l$ projects to an index~0 Morse critical point of a Hamiltonian function. So $\pi(u_1)$ belongs to a moduli space of virtual dimension $-2$, unless it is constant. It implies that $\pi(u_1)$ is constant, therefore $u_1$ lies in a fixed fibre.

Now consider the curves $v_1,\ldots,v_q$ in $X\setminus U(\Sigma)$ attached to the orbits $\gamma_i$. Each $v_i$ is a holomorphic plane, because $u_1$ is the only component of the building that lies in $U(\Sigma)$. Note that $\gamma_i$ is a fixed Reeb orbit, and let $l_i$ be its multiplicity (the winding number around $\Sigma$). Then $v_i$ solves a problem of the same index as that of the following: holomorphic Chern number $l_i$ spheres in $X$, passing through a specified point $p\in \Sigma$ with local intersection multiplicity $l_i$. This problem has index $-2$, hence no solutions exist generically. To rule out multiply-covered solutions, one  checks that the underlying simple curves also have negative virtual dimension. See \cite[Lemma~2.1]{To18} for the case of spheres with a tangency; the punctured SFT case is analogous. 

We have proven that $u_1$ has no punctures. This means that there was no breaking under the SFT limit, and the whole (b)-curve stayed in $U(\Sigma)$. We have also shown that $l=k$.
\end{proof}

\begin{proof}[Proof of Proposition~\ref{prop:bs_antican}]
	One needs to  show that $\BS=r$ and $\D_k=r^k$; we will prove the former since the latter is analogous. Given the previous lemma, there are several ways of completing the argument by a modification of \cite{BO09}. Our strategy is to re-run the proof of Theorem~\ref{th:factor} where instead of holomorphic disks we stretch holomorphic planes in the neighbourhood of $\Sigma$ whose count is known to be $1$.

	Let $U(\Sigma)$ be the neighbourhood appearing above, and consider its completion by an infinite concave end: $M=((-\infty,0]\times \bd U(\Sigma))\cup U(\Sigma)$; see \cite{CompSFT03} for the basic terminology. Pick a contact form at the negative boundary $\bd_{-\infty}M$ whose Reeb flow lifts the Hamiltonian flow of a Morse function $h_\Sigma$ on $\Sigma$ with unique minimum. This contact structure is a perturbation of the standard one whose Reeb flow is the 1-periodic rotation around $\Sigma$.
	Denote by $\hat r$  a parameterised Reeb orbit over the minimum of $h_\Sigma$ considered as a degree~0 element of the non-equivariant contact homology complex of $\bd_{-\infty}M$ \cite{BO09}. Recall that $\Sigma$ is anticanonical, so the degree of $\hat r$ is zero by (\ref{eq:CH_chain_level_0}). (We consider this complex just as a vector space and are not interested in its differential, whose definition in general can meet certain  difficulties.)
	
	Equip $M$ with an almost complex structure $J$ which is cylindrical near the negative end.  Consider the moduli space $\M(\hat r)$
of $J$-holomorphic planes $u\co \C\to M$ asymptotic to $\hat r$ at infinity considered as a negative puncture, with an asymptotic  marker matching the initial point of $\hat r$, such that $u(0)\in \Sigma$ and the total intersection number with $\Sigma$ equals $1$. This is a rigid problem, since this moduli space has dimension $-|\hat r|=0$. We claim that the count of solutions in this moduli space equals 1.  One way of seeing this is as follows.
Consider an embedding of $M$ into the projectivisation of the normal bundle to $\Sigma$, and let $\tilde M$ be its complement, seen as a domain with convex boundary.
We glue the elements in $\M(\hat r)$ to analogous holomorphic planes living in $\tilde M$ with a positive puncture asymptotic to $\hat r$ (this time modulo $S^1$-reparameterisation) and passing through a fixed point in $\tilde \Sigma$, the infinity-section of the projectivisation. The glued solutions become holomorphic curves in the projectivisation of the normal bundle to $\Sigma$. One homotopes the glued almost complex structure to one for which the projection onto $\Sigma$ is everywhere holomorphic, after which the glued solutions become the unique fibre sphere. 
	
The count of $\M(\hat r)$ has been established. Now one runs the analogue of the proof of	Theorem~\ref{th:factor} for these curves. All arguments carry over; it must be additionally mentioned that SFT breaking at the negative end is impossible because any non-trivial SFT bubble projects to a sphere in $\Sigma$ with positive Chern number (by positivity of area and monotonicity of $X$), making the part of the broken configuration in $M$ belong to a moduli space of virtual dimension $\le -2$ which is generically empty.
The resulting (b)-curves identically compute the  Borman-Sheridan class as defined in Section~\ref{sec:proofs}. The (a)-curves this time do  not compute the closed-open map, but instead   the continuation map $\Psi$ of Bourgeois and Oancea \cite[Section~6]{BO09b} realising the quasi-isomorphism from symplectic to non-equivariant contact homology. This isomorphism is known to be action-non-increasing, and its low-energy ``diagonal term'' is already a quasi-isomorphism (in our case, an isomorphism in degree~0). 
 We now also see that the continuation (a)-curves from $r$ to $\hat r$ are in fact computing a diagonal entry of the Bourgeois-Oancea quasi-isomorphism, so their count is 1. We conclude that the count of the outputs of the (b)-curves, i.e.~the Borman-Sheridan class, precisely $r$.

The same argument works for $\D_k$ by considering the $k$-fold Reeb orbit $\hat r^k$ instead of $\hat r$.
The curves we glue in $\tilde M$ have an order $k$ tangency to a given point in $\tilde\Sigma$, and the resulting holomorphic spheres are $k:1$ branched over the fibre sphere, with unique branching point at $\tilde \Sigma$.
 At the last step of the proof, an action estimate for the (b)-curve, which has intersection number $k$ with $\Sigma$, implies that its output is a combination of generators $r^i$ for $i\le k$. But all these generators except the $r^k$ have lower action than $\hat r^k$, so the existence of the continuation (a)-curve implies that the Hamiltonian orbit has to be $r^k$; the argument concludes as above.
\end{proof}

\subsection{Partial compactifications}
In this subsection we discuss a version of Theorem~\ref{th:factor} in a class of situations where $X$ is allowed to be non-compact. A convenient setting is the following one. Let $Y$ be a compact Fano variety equipped with its monotone K\"ahler symplectic form and $\Sigma\subset Y$ be a \emph{normal crossings anticanonical} divisor. Denote its irreducible components by $$\Sigma=\cup_{i\in I}\Sigma_i.$$
Choose a subset of the irreducible components $J\subset I$ and let
$$
X=Y\setminus (\cup_{i\in J}\Sigma_i).
$$
Suppose $L\subset Y$ is a monotone Lagrangian submanifold which is disjoint from $\Sigma$ and exact in $Y\setminus\Sigma$ with respect to $-d^c\log\| s\|$, where $s$ is a section defining $\Sigma$. Obviously, $L\subset X$ is also monotone and therefore there it has well-defined potentials in $X$ and in $Y$:
$$
W_{L,X},\ W_{L,Y}.
$$
Consider a class $A\in H_2(Y,L;\Z)$ such that $\mu(A)=2$. Then by (\ref{eq:mu_and_intersec}), $A\cdot[\Sigma]=1$.

\begin{lemma}
	\label{lem:intersect_part}
The potential $W_{L,X}$ can be obtained by computing all holomorphic disks in $Y$ that contribute to $W_{L,Y}$ and picking out those disks whose homology classes $A$ have the following property: 
\begin{equation}
\label{eq:disj_s_i}
A\cdot [\Sigma_i]=0 \quad\text{for all}\quad  i\in J.
\end{equation}
\end{lemma}

\begin{proof}
	Clearly, the potential $W_{L,X}$ counts those disks that contribute to $W_{L,Y}$ which are moreover disjoint from $\Sigma_i$ for all $i\in J$. It follows that the classes of such disks satisfy $A\cdot[\Sigma_i]=0$.
	Conversely, if
	$A\cdot[\Sigma_i]=0$,
	then the disks in this class are disjoint from $\Sigma_i$ by positivity of intersections, assuming the chosen almost complex structure preserves $\Sigma$. 
\end{proof}	
	
\begin{theorem}
	\label{th:factor_red}
Let $Y$, $\Sigma$ and $X$ be as above. Suppose 
$$M\subset Y\setminus\Sigma$$
is any Liouville subdomain. 
There exists a class $\BS_{M,X}\in SH^0(M)$ with the following property. For any monotone $L\subset Y$ which is contained in $M$ and exact in $Y\setminus\Sigma$ (automatically, $L$ is exact in $M$), and any local system $\LL=(L,\rho)$, one has:
$$
\CO(\BS_{M,X})=W_{L,X}(\rho)
$$
where $\CO\co M\to HF^*_M(\LL,\LL)$ is the closed-open map.
\end{theorem}

\begin{remark}
The requirements that $c_1(M)=0$ and $M$ is grading-compatible, which appeared in Theorem~\ref{th:factor}, are automatic when $\Sigma$ is anticanonical, so they do not appear here.
\end{remark}	
	
\begin{proof}
Let $\tilde{\Sigma}$ be the smoothing of $\Sigma$, contained in a neighbourhood of $U(\Sigma)$. Let $s_{\Sigma}$, $s_{\tilde \Sigma}$ be the corresponding defining section. One can deform the Liouville structure $-dd^c\log\|s_{\tilde \Sigma}\|$ on $Y\setminus \Sigma$ to match  $-dd^c\log\|s_{\Sigma}\|$ outside of $U(\Sigma)$, so there is an embedding of Liouville domains 
$$
Y\setminus\Sigma\subset Y\setminus\tilde\Sigma,
$$	
 therefore $M$ can be viewed as a Liouville subdomain of $Y\setminus\tilde{\Sigma}$. One may run the proof of Theorem~\ref{th:factor} applied to
$$
L\subset M\subset Y\setminus\tilde \Sigma.
$$
Among the Maslov index 2 disks with boundary on $L$, one picks out those whose homology classes satisfy (\ref{eq:disj_s_i}) as a topological condition. After domain-stretching, the (a)-curves are confined to $M$ therefore have zero intersection with any $\Sigma_i$. So the (b)-curves still have the property that their homological intersection with $[\Sigma_i]$, $i\in J$, vanishes. We define $\BS_{M,X}$ by counting the output asymptotics $\gamma$ of those (b)-curves defining the usual Borman-Sheridan class $\BS$ for $M\subset X\setminus\tilde{\Sigma}$ which satisfy this extra topological condition. The rest of the proof remains unchanged.
\end{proof}		

\begin{example}
Let $L\subset \C^2$ be the monotone Clifford torus which remains monotone after the compactification $$X\coloneqq \C^2(x,y)\,  \subset\,  
Y\coloneqq \C P^2(x,y,z).$$ The space $Y=\C P^2$  plays the role of the compact Fano variety. Consider the anticanonical divisor
$$
\Sigma=\Sigma_1\cup\Sigma_2\subset Y
$$
where $\Sigma_1=\{z=0\}$ is the line at infinity, so that
$$
X=Y\setminus\Sigma_1,
$$
and $\Sigma_2$ is a smooth conic, e.g.~$\{xy=\epsilon z^2\}$, so that
$$
L\subset M\coloneqq \C P^2\setminus\Sigma=\C P^2\setminus(\Sigma_1\cup\Sigma_2)
$$
is exact.  Theorem~\ref{th:factor_red} asserts that there is a class
$$
\BS_{M,X}\in SH^0(M),
$$
which under the closed-open map onto $L$, with various local systems, computes the potential
$$
W_{L,X}(u,v)=(u+1)v
$$
of the Clifford torus in $\C^2$, written in some basis $(u,v)$. Next one can consider the Chekanov torus $L'\subset M\subset \C P^2$ and repeat the story for it (see e.g.~\cite{Pa14,PT17} and \cite[Section~11]{SeiBook13} for details), with the upshot that applying the closed-open map onto $L'\subset M$ makes the same Borman-Sheridan class compute the potential of the Chekanov torus in $\C^2$,
$$
W_{L',X}(u,v)=v^{-1}.
$$
Pascaleff computed \cite{Pa13,Pa14} that
$$
SH^0(M)\cong\C[p,q,(1-pq)^{-1}].
$$
It follows from  \cite[Section~7]{EDT17b} and Remark~\ref{rmk:CO_BS}  that the closed-open maps
to the Clifford and~the Chekanov torus are given respectively by
$$
\left\{
\begin{array}{l}
\CO_L(p)=(u+1)v\\
\CO_L(q)=v^{-1}
\end{array}
\right.
\quad \textit{and}\quad 
\left\{
\begin{array}{l}
\CO_{L'}(p)=v^{-1}\\
\CO_{L'}(q)=(u+1)v.
\end{array}
\right.
$$
It follows from this and Remark~\ref{rmk:CO_BS} that
$$
\BS_{M,X}=p.
$$
This correctly recovers the potentials of both $L$ and $L'$ (the Clifford and the Chekanov torus) under the closed-open map.
\end{example}

\begin{remark}
In view of the example above and the mirror-symmetric context explained in the introduction, it may be an interesting problem to look for relationships between the symplectic cohomology of the complement of a normal crossings divisor $SH^0(Y\setminus\Sigma)$, and the LG potentials of monotone tori in the partial compactifications $X$ of $Y\setminus\Sigma$. (One hopes for a statement in the spirit of Corollary~\ref{cor:struc_coef}, except that now the usual LG potentials may already give interesting relations in symplectic cohomology rings.)
\end{remark}
	
\bibliography{Symp_bib}{}
\bibliographystyle{plain}

\end{document}